\documentclass[11pt]{amsart}

\usepackage{amssymb,amsmath,amscd}
\usepackage{graphicx}
\usepackage{bbm}
\usepackage{a4wide}

\theoremstyle{plain}
\newtheorem{theorem}{Theorem}
\newtheorem{prop}[theorem]{Proposition}
\newtheorem{lemma}[theorem]{Lemma}
\newtheorem{coro}[theorem]{Corollary}

\theoremstyle{definition}
\newtheorem{remark}{Remark}

\newcommand{\ts}{\hspace{0.5pt}}
\newcommand{\nts}{\hspace{-0.5pt}}

\newcommand{\CC}{\mathbb{C}\ts}
\newcommand{\RR}{\mathbb{R}\ts}
\newcommand{\QQ}{\mathbb{Q}\ts}
\newcommand{\ZZ}{{\ts \mathbb{Z}}}

\newcommand{\NN}{\mathbb{N}}
\newcommand{\XX}{\mathbb{X}}

\newcommand{\cA}{\mathcal{A}}
\newcommand{\cB}{\mathcal{B}}

\newcommand{\cI}{\mathcal{I}}

\newcommand{\cL}{\mathcal{L}}
\newcommand{\cF}{\mathcal{F}}

\newcommand{\vL}{\varLambda}
\newcommand{\one}{\mathbbm{1}}
\newcommand{\ii}{\mathrm{i}}
\newcommand{\ee}{\mathrm{e}}

\newcommand{\dd}{\, \mathrm{d}}

\newcommand{\oplam}{\mbox{\Large $\curlywedge$}}
\newcommand{\bs}{\boldsymbol}

\DeclareMathOperator{\dens}{dens}
\DeclareMathOperator{\vol}{vol}
\DeclareMathOperator{\Mat}{Mat}

\DeclareMathOperator{\card}{card}
\DeclareMathOperator{\GL}{GL}
\DeclareMathOperator{\sinc}{sinc}
\DeclareMathOperator{\supp}{supp}

\begin{document}

\title[Pair correlations via renormalisation]
{Pair correlations of aperiodic inflation rules\\[2mm]
 via renormalisation:\ Some interesting examples}

\author{Michael Baake}

\author{Franz G\"{a}hler}
\address{Fakult\"at f\"ur Mathematik, Universit\"at Bielefeld, \newline
\hspace*{\parindent}Postfach 100131, 33501 Bielefeld, Germany}
\email{$\{$mbaake,gaehler$\}$@math.uni-bielefeld.de}

\begin{abstract}
  This article presents, in an illustrative fashion, a first step
  towards an extension of the spectral theory of constant length
  substitutions. Our starting point is the general observation that
  the symbolic picture (as defined by the substitution rule) and its
  geometric counterpart with natural prototile sizes (as defined by
  the induced inflation rule) may differ considerably. On the
  geometric side, an aperiodic inflation system possesses a set of
  exact renormalisation relations for its pair correlation
  coefficients. Here, we derive these relations for some paradigmatic
  examples and infer various spectral consequences. In particular, we
  consider the Fibonacci chain, revisit the Thue--Morse and the
  Rudin--Shapiro sytem, and finally analyse a twisted extension of the
  silver mean chain with mixed singular spectrum.
\end{abstract}

\maketitle

\section{Introduction}

The spectral theory of primitive substitution rules and the symbolic
dynamical systems defiend by them has a long history and is well
studied, compare \cite{Q} and references therein. Nevertheless, as
soon as one leaves the realm of constant length substitutions, many
open questions challenge our present level of understanding.  This is
partly due to the fact that the substitution system picture and the
inflation tiling picture, where one works with natural prototile
lengths, may differ considerably; compare \cite{CS,Robbie,TAO,BL-2}.

The majority of the dynamical systems literature on this subject deals
with the symbolic case, where work by Dekking \cite{Dekking} and
others, see \cite{NF1,NF2,FS,LMS2} and references therein, has led to
a fairly complete understanding of the constant length case, which was
recently extended in a systematic fashion in \cite{Bart}, including
higher dimensions. In all these cases, the connection between the
diffraction spectrum and the dynamical spectrum is rather well
understood. The two notions are equivalent in the pure point case
\cite{LMS,BL,BLM,LS}, and recent progress for symbolic systems also
establishes a complete equivalence via the inclusion of certain
factors \cite{BLvE}. Therefore, at least algorithmically, any constant
length substitution can be analysed completely as far as its spectrum
is concerned; see \cite{GTM,squiral,Bart} for recent developments.

The situation is less favourable outside this class. Even if one stays
within the realm of primitive inflation rules, the determination of
the dynamical spectrum is generally difficult. Here, the geometric
counterpart studied in tiling theory has certain undeniable
advantages. First, since such tilings are still of finite local
complexity, the connection between the dynamical spectrum and the
diffraction spectrum can still be used, if certain factors are
included in the discussion \cite{BLvE}. For this reason, we employ an
approach via the diffraction spectrum of the tiling spaces. Second,
the geometric setting leads to the the existence of a set of
\emph{exact} renormalisation relations for the pair correlation
coefficients of the system.  We use the term `exact' to distinguish
our approach from the widely used renormalisation schemes that are
approximate or asymptotic in nature. This approach has not attracted
much attention so far.  In fact, we are not aware of \emph{any}
reference to this approach beyond the constant length case, though the
principal idea has certainly been around for a while, and has been
used, often in an approximate way, for several physical quantities
such as electronic or transport properties; see \cite{TS} for an
example and further references.

In this paper, we we want to show the power of exact renormalisation
relations for spectral properties, in the form of a first step via
some illustrative examples. A more general approach will be presented
in a forthcoming publication.  It will be instrumental for our
analysis that we formulate various aspects on the symbolic level,
while the core of our analysis rests on the natural geometric
realisation. To make the distinction as transparent as possible, we
will speak of \emph{substitution rules} on the symbolic side, but of
\emph{inflation rule} on its geometric counterpart, thus following the
notation and terminology of the recent monograph \cite{TAO}. Also,
various results are briefly recalled from there. Rather than repeating
the proofs, we provide precise references instead.

Below, we work along a number of examples, all of them with
Pisot--Vijayaraghavan (PV) numbers as inflation multipliers. Since we
do \emph{not} demand the characteristic polynomial to be irreducible
over $\QQ$, there is still enough freedom to encounter systems with
mixed spectrum. In fact, all point sets that we encounter along the
way will be linearly repetitive Meyer sets. An interesting linearly
repetitive inflation point set with non-PV multiplier (hence not a
Meyer set) will be discussed in detail in \cite{BFGR}.

The paper is structured as follows. After recalling some facts about
translation bounded measures on $\RR$ and their Fourier transforms in
Section~\ref{sec:prelim}, we begin with a detailed analysis of the
Fibonacci substitution and inflation in Section~\ref{sec:Fibo}.  This
is both a paradigm of the theory and an instructive example along
which we can develop our ideas as well as further notions, wherefore
this is also the longest section. Here, we introduce a set of exact
renormalisation equations for the general pair correlation
coefficients and compare the findings with what is known from the
model set description. Then, we use the new approach to derive an
alternative proof of the pure point nature of the diffraction
spectrum, and hence also of the dynamical spectrum via the known
equivalence result for this case \cite{LMS,BL,LS}.

After that, in Section~\ref{sec:TM-RS}, we briefly revisit the classic
Thue--Morse and Rudin--Shapiro sequences from the renormalisation
point of view. The point of this exercise is to show that our
approach, in the constant length setting where the symbolic and the
geometric pictures coincide, is essentially equivalent to the
traditional approach as described in \cite{Q} and recently extended in
\cite{Bart}. Still, there are several aspects of a more algebraic
nature that seem to deserve further attention.

Finally, by imposing an involutory \emph{bar swap} symmetry, we
construct an extension of the silver mean chain with mixed spectrum in
Section~\ref{sec:sm-mixed}. The main point here is that such
extensions are not restricted to the constant length case (where they
are known from examples such as those of the previous section or many
others as in \cite{GTM,BGG,squiral}).  In fact, as our example
indicates, there is an abundance of interesting and completely natural
primitive inflation rules that produce repetitive Meyer sets with
mixed spectrum. In our opinion, this has hitherto been more or less
neglected. For the explicit analysis, the exact renormalisation scheme
is used to determine the spectral type by a somewhat subtle
application of the Riemann--Lebesgue lemma, which leads to a singular
spectrum of mixed type.

Overall, we believe that the geometric picture with its exact
renormalisation relations for the pair correlation coefficients
provides a powerful tool also for the general case of primitive
inflation rules. This will be further developed in a forthcoming
publication.

\section{Preliminaries}\label{sec:prelim}

A \emph{measure} on $\RR$ is a continuous linear functional on the
space $C_{\mathsf{c}} (\RR)$ of continous functions with compact
support. By the Riesz--Markov theorem, this specifies precisely the
class of regular Borel measures on $\RR$. Note that these need not be
finite measures, though all examples below will be \emph{translation
  bounded}, meaning measures $\mu$ such that the total variation
measure $\lvert \mu \rvert$ satisfies $\sup_{t\in\RR} \lvert \mu
\rvert (t+K) < \infty$, for any compact $K\subset \RR$; see \cite{Hof}
and \cite[Sec.~8.5]{TAO} for more. If $\mu$ is a measure, its twisted
counterpart $\widetilde{\mu}$ is defined via $\widetilde{\mu} (g) =
\overline{\mu (\widetilde{g}\ts )}$ for $g\in C_{\mathsf{c}} (\RR)$,
where $\widetilde{g} (x) := \overline{g (-x)}$.

The convolution of two finite measures is denoted as $\mu * \nu$, and
the volume-averaged (or Eberlein) convolution of two translation
bounded measures by
\[
    \mu \circledast \nu \, := \, \lim_{R\to\infty}
    \frac{\mu |^{}_{R} * \nu |^{}_{R} }{2 R} \ts ,
\]
provided the limit exists (we shall not consider any other case
below). Here, $\mu |^{}_{R}$ denotes the restriction of $\mu$ to the
interval $(-R,R)$.

A measure $\mu$ is called \emph{positive definite}, if $\mu (g *
\widetilde{g}\ts ) \ge 0 $ holds for all $g\in C_{\mathsf{c}} (\RR)$.
There are several possibilities to define the Fourier transform of a
measure, provided it exists. This is a non-trivial issue, and part of
our later analysis will rely on the existence. Here, we use a version
of the Fourier transform that, for integrable functions on $\RR$,
reduces to
\[
    \widehat{f} (k) \, = \int_{\RR} \ee^{-2 \pi \ii kx} 
     f(x) \dd x  \ts .
\]
Any positive definite measure is Fourier transformable. The Fourier
transform of a positive definite measure is a positive measure; see
\cite[Ch.~I.4]{BF} or \cite[Sec.~8.6]{TAO} for details and
\cite{Rudin, BF} for general background.

\begin{lemma}\label{lem:transformable}
  Let\/ $\mu$, $\nu$ be translation bounded measures such that\/
  $\mu\circledast \widetilde{\nu}$ as well as\/
  $\mu\circledast\widetilde{\mu}$ and\/
  $\nu\circledast\widetilde{\nu}$ exist. Then, $\mu\circledast
  \widetilde{\nu}$ is a translation bounded and transformable measure,
  as is\/ $\widetilde{\mu}\circledast\nu$.
\end{lemma}
\begin{proof}
  Observe first that $\widetilde{\mu}\circledast\nu = \widetilde{\mu
    \circledast\widetilde{\nu}}$, wherefore it suffices to prove the
  claim for $\mu \circledast\widetilde{\nu}$. Now, as a variant of the
  (complex) polarisation identity, one verifies that
\begin{equation}\label{eq:polar}
   \mu\circledast \widetilde{\nu} \, = \, \frac{1}{4}
   \sum_{\ell=1}^{4} \ii^{\ell} (\mu + \ii^{\ell} \nu) \circledast
   (\mu + \ii^{\ell} \nu)\!\widetilde{\phantom{T}} ,
\end{equation}
where all measures on the right hand side exist due to our
assumptions.  Consequently, $\mu\circledast \widetilde{\nu}$ is a
complex linear combination of four positive definite measures, each of
which is translation bounded and transformable, hence $\mu\circledast
\widetilde{\nu}$ is translation bounded and transformable as well.
\end{proof}

Let us now briefly recall an important result on measures with Meyer
set support, which is due to Strungaru \cite{Nicu}.  Let $\vL$ be a
repetitive Meyer set, and assume that the hull of $\vL$ is uniquely
ergodic under the translation action of $\RR$.  Then, the
autocorrelation measure $\gamma$ of $\vL$ is well-defined and has a
unique Eberlein decomposition as
\[
    \gamma \, = \, \gamma^{}_{\mathsf{s}} + \gamma^{}_{0}
\]
into a strongly almost periodic measure, whose Fourier transform
exists and is a pure point measure, and a null-weakly almost periodic
measure, whose Fourier transform is a continuous measure; see
\cite{GLA} for the underlying notions and results.

Moreover, both parts, $\gamma^{}_{\mathsf{s}}$ and $\gamma^{}_{0}$,
are supported in a model set with a closed, compact window.  Given a
cut and project scheme for $\vL$ according to \cite{BM}, one can use
the smallest such window that defines a model set cover of $\vL -
\vL$.  The same conclusion holds for the auto\-correlation of a weighted
extension.  For our slightly more general situation, where we do not
only consider measures of the form $\mu \circledast \widetilde{\mu}$
but also products of the form $\mu \circledast \widetilde{\nu}$ as in
Lemma~\ref{lem:transformable}, one employs once more the complex
polarisation identity~\eqref{eq:polar} to obtain the corresponding
decomposition result for the more general correlation measures that we
need.

\section{Example 1: The Fibonacci chain}\label{sec:Fibo}

One version of the ubiquitous Fibonacci substitution is given
by the rule $\varrho \! : \, a \mapsto ab \ts , \, b \mapsto a$.
It is primitive, with substitution matrix 
\[
      M \, = \, M^{}_{\varrho} = \begin{pmatrix}
      1 & 1 \\ 1 & 0 \end{pmatrix}
\]
and characteristic polynomial $x^2 - x - 1$. This polynomial is
irreducible over $\QQ$, with roots $\tau$ and $\tau ' = 1 - \tau$,
where $\tau = \frac{1}{2} (1 + \sqrt{5}\,)$ is the golden ratio;
see \cite{TAO} for details and background.

For a geometric realisation, one reinterprets the letters as marked
intervals in $\RR$, with a point at their left endpoints. Choosing
length $\tau$ for $a$ and length $1$ for $b$, which emerges from
the left Perron--Frobenius (PF) eigenvector of $M$, compare
\cite[Ch.~4]{TAO} for background, one obtains the (geometric) Fibonacci
inflation rule. Here, in one iteration step, each interval is
streched by a factor of $\tau$ and then dissected into intervals
of the original size according to the symbolic rule, which is
possible due to the eigenvector condition. By abuse of notation,
we also call this inflation rule $\varrho$.

A bi-infinite fixed point is obtained via iterated inflation starting
from a legal seed, $a|a$ say, where $|$ marks the reference point that
is placed at $0$ in the geometric realisation; see \cite[Ch.~4]{TAO}
for the basic definitions. More precisely, the
corresponding iteration sequence converges (in the local topology)
towards a $2$-cycle under the inflation. Each member of that cycle is
then a \emph{fixed point} under $\varrho^2$. It is a tiling of $\RR$
with two prototiles, each of which carries a point at its left
endpoint.  Consequently, one may equally well view it as a Delone set
in $\RR$ that carries a colouring according to the types, $a$ and
$b$. We shall use both pictures in parallel. Let us note in passing
that the two fixed points of $\varrho^{2}$ are \emph{proximal} in a
strong sense. Indeed, they completely agree except for the first two
positions on the left of the marker, where they alternate between $ab$
and $ba$ under the action of $\varrho$.

For concreteness, let us assume that we select the fixed point of
$\varrho^{2}$ with seed $a|a$. Then, for the (coloured) Delone set
$\vL$ of its marker points, we obtain the relation
\[
    \vL \, = \, \vL_{a} \ts \dot{\cup} \ts \vL_{b} \ts ,
\]
where each of the three sets is a regular model set (or
cut and project set) for the cut and project scheme (CPS), see
\cite[Sec.~7.2]{TAO},
\begin{equation}\label{eq:candp}
\renewcommand{\arraystretch}{1.2}\begin{array}{r@{}ccccc@{}l}
   & \RR & \xleftarrow{\,\;\;\pi\;\;\,} & \RR \times \RR & 
        \xrightarrow{\;\pi^{}_{\mathrm{int}\;}} & \RR & \\
  \raisebox{1pt}{\text{\footnotesize dense}\,} \hspace*{-1ex}
   & \cup & & \cup & & \cup & \hspace*{-1ex} 
   \raisebox{1pt}{\,\text{\footnotesize dense}} \\
   & \ZZ[\tau] & \xleftarrow{\; 1-1 \;} & \cL & 
        \xrightarrow{\; 1-1 \;} &\ZZ[\tau] & \\
   & \| & & & & \| & \\
   & L & \multicolumn{3}{c}{\xrightarrow{\qquad\qquad\;\,\star
       \,\;\qquad\qquad}} 
       &  {L_{}}^{\star\nts}  & \\
\end{array}\renewcommand{\arraystretch}{1}
\end{equation}
with $\ZZ[\tau]$ being the ring of integers in the real quadratic
field $\QQ(\sqrt{5}\,)$, while
\[
   \cL \, = \, \bigl\{ (x,x^{\star}) \mid x \in
       \ZZ[\tau] \bigr\} \subset \RR^{2}
\]
is the lattice that emerges from the ring $\ZZ[\tau]$ via its Minkowski
embedding; see \cite[Fig.~3.3]{TAO}. Here, ${}^{\star}$ denotes
algebraic conjugation in the quadratic field $\QQ(\sqrt{5}\,)$, as
defined by $\sqrt{5} \mapsto - \sqrt{5}$. This is the the $\star$-map
of the CPS of Eq.~\eqref{eq:candp}, with $\tau^{\star} =
\frac{-1}{\tau} = 1-\tau$. Let us stress that the particular choice of
the fixed point has no influence on the CPS --- other choices would
give the same setting; see \cite{BM} for a general result in this
direction. Our choice of fixed point, however, does result in a
specific window for the model set description of $\vL$.

A \emph{model set} in the CPS of Eq.~\eqref{eq:candp} is a subset
of \emph{direct} space $\RR$ of the form
\[
    \oplam (W) \, := \, \{ x \in \ZZ[\tau] \mid x^{\star} \in W \} \ts ,
\]
where the window $W \subset \RR $ is assumed to be a relatively
compact set in \emph{internal} space with non-empty interior. A model
set is called \emph{regular} when $\partial W$ has zero measure in
internal space.  In particular, one finds \cite[Ex.~7.3]{TAO} for our
selected fixed point with seed $a|a$ that
\[
   \vL_{a} \, = \, \oplam \bigl((\tau-2,\tau-1]\ts \bigr) , \quad
   \vL_{b} \, = \, \oplam \bigl((-1,\tau-2]\ts \bigr) 
   \quad \text{and} \quad
   \vL \, = \, \vL_{a} \cup \vL_{b} \, = \,
    \oplam \bigl((-1,\tau-1]\ts \bigr).
\]
The other fixed point of $\varrho^{2}$, with seed $b|a$, can be
described similarly, with the windows now containing the left endpoint,
but not the right one. Model sets are Meyer sets, which means that $\vL$
is relatively dense while $\vL - \vL$ is uniformly discrete; see
\cite[Prop.~7.5]{TAO}.

The natural \emph{autocorrelation} $\gamma^{}_{\vL}$ of the model set
$\vL$ is a translation bounded pure point measure of the form
$\gamma^{}_{\vL} = \sum_{z\in\vL - \vL} \eta (z) \ts \delta_{z}$ with
\[
    \eta(z) \, := \, \lim_{R\to\infty} 
    \frac{\card \bigl( \vL^{}_{R} \cap (z+\vL^{}_{R}) 
    \bigr)}{2R} \ts ,
\]
where $\vL^{}_{R} := \vL\cap (-R,R)$.  The limit exists for all
$z\in\RR$, but is non-zero only for $z\in\vL -\vL$, the latter being a
Delone set in this case (and, in fact, also a model set). Note
that $\eta(-z) = \eta(z)$ holds for all $z\in\RR$.  For $z\in \QQ
(\sqrt{5}\,)$, one obtains from \cite[Prop.~9.8]{TAO} the explicit
formula
\begin{equation}\label{eq:eta-int}
   \eta (z) \, = \, \dens (\vL) \,
   \frac{\vol \bigl(W \cap (z^{\star} + W) \bigr)}{\vol (W)}
   \, = \, \frac{1}{\sqrt{5}}  \int_{\RR} 1^{}_{W} (y) \ts
   1^{}_{z^{\star}+W} (y) \dd y \ts ,
\end{equation}
with $W=(-1,\tau-1]$ as above and $\dens (\vL) = \tau/\sqrt{5}$.
Similar formulas emerge for the subsets $\vL_{a}$ and $\vL_{b}$.
Since $\vL - \vL \subset \ZZ[\tau]$, all non-zero values of $\eta (z)$
are covered by formula~\eqref{eq:eta-int}.

The measure $\gamma^{}_{\vL}$ is positive definite, and hence Fourier
transformable. By the Bochner--Schwartz theorem,
$\widehat{\gamma^{}_{\vL}}$ is then a positive measure, called the
\emph{diffraction measure} \cite{Hof,TAO} of the set $\vL$. Here, it
is a pure point measure of the form
\[
    \widehat{\gamma^{}_{\vL}} \, = \, \sum_{k\in \cF}
    I(k) \ts \delta^{}_{k} \ts , \quad \text{with} \quad
    I(k) \, = \, \left( \frac{\tau}{\sqrt{5}} \, 
    \sinc (\pi \tau k^{\star}) \right)^{2} ,
\]
where $\cF = \ZZ[\tau]/\sqrt{5} \subset \QQ (\sqrt{5}\,)$ and $\sinc
(x) := \frac{\sin(x)}{x}$; see \cite[Sec.~9.4.1]{TAO} for
details. Note that $I(k) =0$ for $k= m\tau$ with $m\in\ZZ\setminus \{
0 \}$, which are all points of $\cF$ for which the intensity
vanishes. Such points are known as the \emph{extinction points} of the
diffraction measure (or of the Fourier--Bohr spectrum); compare
\cite[Rem.~9.10 and Sec.~9.4.1]{TAO}.

The diffraction measure was calculated for the specific set $\vL$,
but it is the same for all other members of the hull of $\vL$,
\[
    \XX (\vL) \, := \, \overline{ \{ t + \vL \mid
       t \in \RR \} } \ts ,
\]
where the closure is taken in the local topology as usual \cite{TAO}.
We thus know that the diffraction measure $\widehat{\gamma^{}_{\vL}}$
is that of the hull as well, and hence that of the dynamical system
$\bigl(\XX (\vL), \RR\bigr)$ with the canonical translation action of
$\RR$.  This system is strictly ergodic \cite{Q,Sol,MR}, so it is
minimal and possesses a unique invariant probability measure,
$\mu^{}_{\mathrm{F}}$ say.  By the general eqivalence theorem between
pure point diffraction and dynamical spectra, see
\cite{Martin,LMS,BL,LS} and references therein, we then also know that
the dynamical system $\bigl(\XX (\vL), \RR, \mu^{}_{\mathrm{F}}\bigr)$
has pure point spectrum. Here, the \emph{dynamical spectrum}
(in additive formulation) is given by
\begin{equation}\label{eq:Fib-spec}
   \cF \, = \, \ZZ[\tau]/\sqrt{5} \ts ,
\end{equation}   
which is the smallest additive subgroup of $\RR$ that contains all
points (wave numbers) $k$ with $I(k) > 0$; see \cite{BL,Lenz,BLvE} and
references therein for further background.  \medskip

Let us now refine the analysis a little, in that we not only look at
the autocorrelation of $\vL$ as a point set, but as a \emph{coloured}
point set. To do so, let us define the pair correlation functions (or
coefficients) $\nu^{}_{\alpha \beta} (z)$ with $\alpha,\beta \in \{
a,b\}$ as the \emph{relative} frequency of two points in $\vL$ at
distance $z$ subject to the condition that the left point is of type
$\alpha$ and the right one of type $\beta$. The term `relative' means
that the frequency be determined per point of $\vL$ rather than per
unit length. Clearly, $\nu^{}_{\alpha \beta} (z)=0$ for any $z\not\in
\vL-\vL$, and $\nu^{}_{\alpha \beta} (0)=0$ for $\alpha\neq\beta$ (as
the type of a point is unique). In fact, the right PF eigenvector of
the substitution matrix $M$ gives us the values $\nu^{}_{aa} (0) =
1/\tau$ and $\nu^{}_{bb} = 1/\tau^2$, which coincide with the relative
frequencies of the letters $a$ and $b$ in the (symbolic) fixed point.

As follows once again from the model set description of $\vL$ (or
alternatively from the unique ergodicity of the hull under the
translation action of $\RR$ via the ergodic theorem), the refined pair
correlation coefficients exist. One has the general symmetry relation
\begin{equation}\label{eq:symm}
   \nu^{}_{\alpha \beta} (-z) \, = \, \nu^{}_{\beta \alpha} (z) \ts ,
\end{equation}
which can be verified directly from the definition of the coefficients
as a mean by a simple calculation. Employing the above model set
description, the coefficients are given by
\begin{equation}\label{eq:nu-sum}
    \nu^{}_{\alpha \beta} (z) \, = \,
    \frac{\vol \bigl( W_{\alpha} \cap 
    (W_{\beta} - z^{\star})\bigr)}{\vol (W)}
\end{equation}
with $W_{a}=(\tau-2,\tau-1]$, $W_{b}=(-1,\tau-2]$ and $W=W_{a}\cup
W_{b}= (-1,\tau-1]$.  One also has the summatory relationship
\begin{equation}\label{eq:rel-nu}
   \frac{ \eta (z)}{\dens (\vL)} \, = \, 
    \nu^{}_{aa} (z) +  \nu^{}_{ab} (z) +  
    \nu^{}_{ba} (z) +  \nu^{}_{bb} (z) \ts ,
\end{equation}
which is illustrated in Figure~\ref{fig:eta-sum} as a relation in
internal space, hence as a function of the variable $z^{\star}$. More
precisely, looking at the right-hand side of Eq.~\eqref{eq:nu-sum} as
a function of $z^{\star}$ for each of the choices $\alpha, \beta \in
\{ a,b \}$ leads to four functions with a continuous extension. Two of
them are the tent-shaped covariograms of $W_{a}$ and $W_{b}$; compare
\cite[Rem.~9.8]{TAO}. They are shaded in Figure~\ref{fig:eta-sum}, one
being shifted in vertical direction. The other two are trapezoidal
functions. When adding them up as shown in the figure, one obtains the
covariogram of the total window $W$.

\begin{figure}[t]
\begin{center}
 \includegraphics[width=0.7\textwidth]{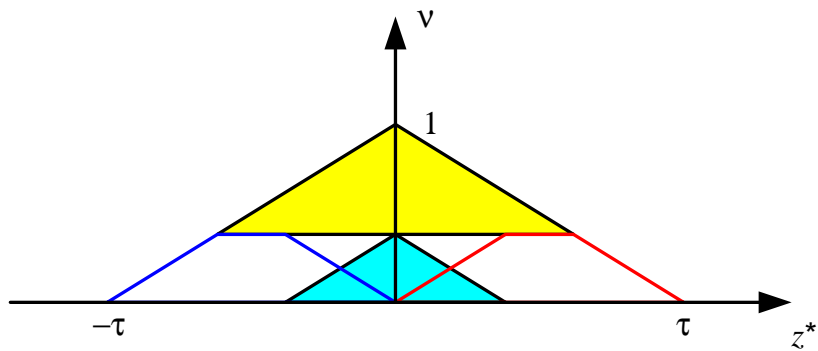}
 \end{center}
 \caption{\label{fig:eta-sum} Representation of the relative
   autocorrelation $\nu$ for the Fibonacci system as a sum of four
   terms in the embedding formalism (internal space) of the
   corresponding CPS; see text around Eq.~\eqref{eq:rel-nu} for
   details.}
\end{figure}

In view of the symmetry relation \eqref{eq:symm}, it suffices to
know the functions $\nu^{}_{\alpha \beta} (z)$ for non-negative
$z\in\vL-\vL$. Note that $\vL - \vL = \vL' - \vL'$ holds for any
$\vL' \in \XX (\vL)$, so that this Minkowski difference is constant
on the hull. Now, the inflation structure tells us how 
the values of $\nu^{}_{\alpha \beta} (z)$ relate to the values
on a scale that is reduced by a factor of $\tau$. 

\begin{prop}\label{prop:Fibo-relations}
  Consider the Fibonacci chain in the geometric setting with interval
  prototiles of lenghts\/ $\tau$ and\/ $1$ as described above. Then,
  the pair correlation coefficients\/ $\nu^{}_{\!\alpha\beta} (z)$,
  with\/ $\alpha, \beta \in \{ a,b\}$, satisfy the exact
  renormalisation relations
\begin{equation}\label{eq:Fibo-rec}
\begin{split}
   \nu^{}_{aa} (z) \, & = \, \tfrac{1}{\tau} \bigl(
     \nu^{}_{aa} (\tfrac{z}{\tau}) + \nu^{}_{ab} (\tfrac{z}{\tau})
    +\nu^{}_{ba} (\tfrac{z}{\tau}) + \nu^{}_{bb} (\tfrac{z}{\tau}) 
       \bigr) , \\
   \nu^{}_{ab} (z) \, & = \, \tfrac{1}{\tau} \bigl(
      \nu^{}_{aa} (\tfrac{z}{\tau} - 1) + 
      \nu^{}_{ba} (\tfrac{z}{\tau} - 1)  \bigr), \\
   \nu^{}_{ba} (z) \, & = \, \tfrac{1}{\tau} \bigl(
      \nu^{}_{aa} (\tfrac{z}{\tau} + 1) + 
      \nu^{}_{ab} (\tfrac{z}{\tau} + 1) \bigr), \\
   \nu^{}_{bb} (z) \, & = \, \tfrac{1}{\tau} \bigl( 
      \nu^{}_{aa} (\tfrac{z}{\tau}) \bigr),
\end{split}
\end{equation}
with\/ $z\in \ZZ[\tau]$ and\/ $\nu^{}_{\!\alpha\beta} (z) = 0$
whenever\/ $z\not\in \varLambda_{\beta} - \varLambda_{\alpha}$.
\end{prop}

\begin{proof}
  Since our inflation rule is aperiodic, we have local recognisability
  \cite{Q}. This means that, in any element of the hull, each tile
  lies inside a unique level-$1$ supertile that is identified by a
  local rule.  Concretely, each patch of type $ab$ constitutes a
  supertile of type $a$, while each tile of type $a$ that is followed
  by another $a$ (to the right) stands for a supertile of type
  $b$. From now on, we simply say supertile, as no level higher than
  $1$ will occur in this proof.

  The distance $z$ of the left endpoints of two tiles is now given
  by the distance $z'$ of the left endpoints of the two supertiles 
  containing them, plus a correction coming from the offsets of the 
  tiles within their supertiles. We may have $z=z'$ (when both tiles 
  are of the same type), but also $z=z'+\tau$ (when the left tile is
  of type $a$ and the right one of type $b$), or $z=z'-\tau$ (for 
  the opposite order); see Figure~\ref{fig:recog} for an illustration.

  In what follows, we first assume $z>0$, and discuss the remaining
  cases afterwards.  Due to the inflation structure, it is clear that
  the frequency of two supertiles of type $\alpha$ (left) and $\beta$
  (right) at distance $z'$, determined relative to the tiles of the
  \emph{original} size, is given by $\frac{1}{\lambda}\ts
  \nu_{\alpha\beta}\bigl( \frac{z'}{\lambda}\bigr)$. This follows from
  the simple observation that, for any $\vL'\in\XX(\vL)$, the point
  set of the left endpoints of the supertiles is a set of the form
  $\tau \vL''$ for some $\vL''\in\XX(\vL)$.

  Collecting all supertile distance frequencies that contribute to a
  given coefficient $\nu^{}_{\!\alpha\beta} (z)$ then results in the
  claimed identities. For instance, $\nu^{}_{bb} (z)$ equals the
  frequency of two type-$a$ supertiles at the same distance, counted
  relative to the original tile sizes, hence giving the last identity
  of \eqref{eq:Fibo-rec}, while $\nu^{}_{\nts aa} (z)$ has
  contributions from all pairs of supertiles at distance $z$.

  The claim on the supports is obvious, and one can check that the set
  of equations properly extends to $z<0$ via $\nu^{}_{\!\alpha\beta}
  (-z) = \nu^{}_{\nts\beta\alpha} (z)$. Next, one can verify explicitly
  that the identities are also satisfied for $z=0$, where they boil
  down to Eq.~\eqref{eq:reno-zero} below, because $ \nu^{}_{\nts ab}
  (0) = \nu^{}_{ba} (0) = 0$ as well as $\nu^{}_{\nts aa} (\pm 1) =
  \nu^{}_{ba} (-1) = \nu^{}_{\nts ab} (1) = 0$ due to the geometry of
  the tiles.  In particular, $ \nu^{}_{\nts aa} (0)$ and $ \nu^{}_{bb}
  (0)$ are the relative prototile frequencies, in agreement with our
  setting.
\end{proof}

\begin{figure}[t]
 \begin{center}
 \includegraphics[width=0.5\textwidth]{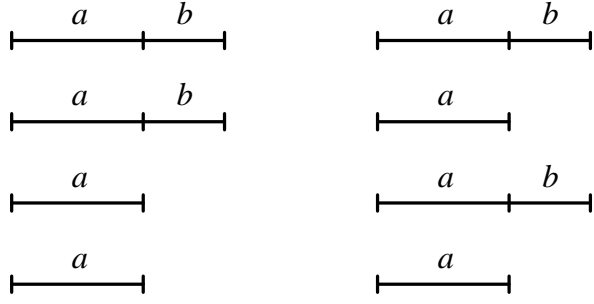}
 \end{center}
 \caption{\label{fig:recog} Pair structure for the derivation of
   Eq.~\eqref{eq:Fibo-rec} via local recognisability. Each of the two
   tiles at distance $z$ is to be located in its unique cover by a
   level one supertile (left versus right column).}
\end{figure}

{}From the structure of the arguments, and recalling that $\tau^{2} =
\tau + 1$, it is immediately clear that all coefficients with
arguments $z$ that satisfy $|z| > \tau^{2}$ are completely determined
in a recursive manner from those with $|z| \le \tau^{2}$, together
with $\nu^{}_{\alpha\beta} (z) = 0$ for $z\notin \vL - \vL$. This is
the \emph{recursive} part of the renormalisation relations.  Likewise,
the relations for all arguments with $|z|\le \tau^{2}$ are closed, and
provide finitely many linear equations to determine the
$\nu^{}_{\alpha\beta} (z) = 0$ for arguments in $[-\tau^{2},\tau^{2}]
\cap (\vL - \vL)$, which is the \emph{self-consistency} part of the
renormalisation relations. As mentioned in the Introduction, we use the
term \emph{exact} to demarcate the relations from similar (and
widely used) concepts that are approximate or asymptotic in nature.

Let us begin with the latter, where we may further restrict ourselves
to non-negative values of $z$ due to Eq.~\eqref{eq:symm}. Observe that
\[
    (\vL-\vL) \cap [0,\tau+1] \, = \, \{ 0, 1, \tau, \tau+1 \} \ts ,
\]
which are thus the positions we have to consider first.  Inserting the
last relation of Eq.~\eqref{eq:Fibo-rec} into the first at $z=0$
implies $\nu^{}_{ab} (0) + \nu^{}_{ba} (0) = 0$, which gives
$\nu^{}_{ab} (0) = \nu^{}_{ba} (0) = 0$ via Eq.~\eqref{eq:symm}.  The
first relation yields $\nu^{}_{aa} (1) = 0$ becaue $\frac{1}{\tau}
\not\in \vL-\vL$. This implies $\nu^{}_{ab} (1) = 0$ via the third
relation, which also gives $\nu^{}_{aa} (-1) = \nu^{}_{ba} (-1) = 0$
by symmetry.  Since $\frac{1}{\tau}\not\in \vL-\vL$, the last
relation, at $z=1$, gives $\nu^{}_{bb} (1) = 0$, whence $\nu^{}_{ba}
(1)$ is the only correlation coefficient at $1$ still to be accounted
for.

The second relation implies $\nu^{}_{ab} (\tau) = \frac{1}{\tau}
\nu^{}_{aa} (0)$, while the first gives $\nu^{}_{aa} (\tau) =
\frac{1}{\tau} \nu^{}_{ba} (1)$. On the other hand, via the third
relation, we see that $ \nu^{}_{ba} (1) = \frac{1}{\tau} \bigl(
\nu^{}_{aa} (\tau) + \nu^{}_{ab} (\tau)\bigr) $, and hence, via
inserting the previous expressions for $\nu^{}_{ba} (1)$ and
$\nu^{}_{ab} (\tau)$, also that
\[
    \nu^{}_{aa} (\tau) \, = \, \tfrac{1}{\tau + 1} \bigl(
    \nu^{}_{aa} (\tau)  + \tfrac{1}{\tau}
    \nu^{}_{aa} (0) \bigr).
\]
This, in turn, gives $\nu^{}_{aa} (\tau) = \frac{1}{\tau^2}
\nu^{}_{aa} (0)$ and thus also $\nu^{}_{ba} (1) = \frac{1}{\tau}
\nu^{}_{aa} (0)$, so that all values at $z=1$ are now determined.
Since $2\not\in\vL-\vL$ and $\nu^{}_{aa} (1) = 0$, we get $\nu^{}_{ba}
(\tau) = \nu^{}_{bb} (\tau) = 0$, so that also all values at $z=\tau$
are accounted for.  The values $\nu^{}_{aa} (\tau+1)$, $\nu^{}_{ab}
(\tau+1)$ and $\nu^{}_{bb} (\tau+1)$ now follow recursively from
Eq.~\eqref{eq:Fibo-rec}. Finally, the third relation gives
\[
    \nu^{}_{ba} (\tau+1) \, = \, \tfrac{1}{\tau} \bigl(
    \nu^{}_{aa} (\tau+1) + \nu^{}_{ab} (\tau+1) \bigr)
\]
and thus fixes the value $\nu^{}_{ba} (\tau+1)$.  We thus have
determined all coefficients with arguments inside
$[-\tau^{2},\tau^{2}]$.

Note that, at $z=0$, the renormalisation relations reduce to
the eigenvector equation
\begin{equation}
   M \begin{pmatrix} \nu^{}_{aa} (0) \\ \nu^{}_{bb} (0)
   \end{pmatrix} \, = \, \tau \cdot \begin{pmatrix} 
   \nu^{}_{aa} (0) \\ \nu^{}_{bb} (0) \end{pmatrix}
   \label{eq:reno-zero}
\end{equation}
with the substitution matrix $M$ from above. Since the eigenspace of
the PF eigenvalue $\tau$ is one-dimensional, the solution of
Eq.~\eqref{eq:Fibo-rec} depends on a single number only, in line with
what we derived from the fourth relation. In summary, we have the
following result.

\begin{lemma}\label{lem:Fibo-rec}
  The renormalisation relations of Eq.~\eqref{eq:Fibo-rec}, subject to
  the symmetry relation of Eq.~\eqref{eq:symm} and the condition that
  all coefficients\/ $\nu^{}_{\!\alpha \beta} (z)$ vanish for any\/
  $z\not\in \vL-\vL$, have a unique solution once the value of\/
  $\nu^{}_{\! aa} (0) $ is given.  
\end{lemma}
\begin{proof}
  The unique determination of all coefficients $\nu^{}_{\alpha \beta}
  (z)$, as a function of $\nu^{}_{aa} (0)$, for $z\in \vL-\vL$ with
  $\lvert z \rvert \le \tau + 1$ was derived above. For all remaining
  $z\in \vL-\vL$, the recursive structure of Eq.~\eqref{eq:Fibo-rec}
  then provides unique values for the coefficients under the
  assumptions made.
\end{proof}

\begin{remark}
  Note that we could have used a stronger restriction of the support,
  by defining one for each coefficient $\nu^{}_{\!\alpha\beta} (z)$
  separately, namely as $S_{\alpha \beta} = \vL_{\beta} -
  \vL_{\alpha}$, which once again would give the same set for any
  element of the hull.  Indeed, the set $\vL - \vL$ is a common
  superset of these individual supports. Nevertheless, this does not
  make any difference here, as the solution even with the larger
  support is the same, in the sense that the coefficients vanish at
  the additional points. Let us also recall \cite{Q,TAO} that strict
  ergodicity of our dynamical system implies that
  $\nu^{}_{\!\alpha\beta} (z) > 0$ if and only if $z\in
  S^{}_{\nts\alpha\beta}$.
\end{remark}

Let us mention that $\vL-\vL$ is itself a model set for the CPS
in Eq.~\eqref{eq:candp}, where one has $\vL-\vL = \oplam \bigl(
(-\tau,\tau)\bigr)$. Later, we shall need to work with model sets
with compact windows, such as $\oplam\bigl( [-\tau,\tau]\bigr)$.
The previous lemma can be extended as follows.

\begin{prop}\label{prop:Fibo-rec}
  The linear renormalisation relations of Eq.~\eqref{eq:Fibo-rec},
  subject to the symmetry relation of Eq.~\eqref{eq:symm} and the
  condition that all coefficients\/ $\nu^{}_{\alpha \beta} (z)$ vanish
  for any\/ $z\not\in \oplam\bigl([-\tau,\tau]\bigr)$, has a
  one-dimensional solution space. In other words, we get a unique
  solution once the value of\/ $\nu^{}_{aa} (0)$ is given.
\end{prop}
\begin{proof}
  Observe first that $\oplam\bigl([-\tau,\tau]\bigr)\setminus
  (\vL-\vL) = \{ \pm \frac{1}{\tau} \}$, which are two points inside
  the interval $[-\tau^{2},\tau^{2}]$. Due to the symmetry relations,
  we only need to consider what the renormalisation relations of
  Eq.~\eqref{eq:Fibo-rec} impose at $z=\frac{1}{\tau}$.  One finds
  $\nu^{}_{bb} \bigl( \frac{1}{\tau}\bigr) = 0$ from the fourth
  relation, $\nu^{}_{ba} \bigl( \frac{1}{\tau}\bigr) = 0$ from the
  third and $\nu^{}_{aa} \bigl( \frac{1}{\tau}\bigr) = 0$ from the
  first, because the arguments on the right hand sides are then
  outside the set $\oplam\bigl([-\tau,\tau]\bigr)$. Furthermore, one
  obtains
\[
    \nu^{}_{ab} \bigl( \tfrac{1}{\tau}\bigr) \, = \,
     \tfrac{1}{\tau} \bigl( \nu^{}_{aa} \bigl(- \tfrac{1}{\tau}\bigr)
    + \nu^{}_{ba} \bigl( -\tfrac{1}{\tau}\bigr) \bigr)
    \, = \, \tfrac{1}{\tau} \nu^{}_{ab} \bigl( \tfrac{1}{\tau}\bigr) 
\]
by the previous identities and the symmetry relation. Clearly, this
implies $\nu^{}_{ab} \bigl( \frac{1}{\tau}\bigr) = 0$, which brings
us back to the situation of Lemma~\ref{lem:Fibo-rec} and the claim
is proved.
\end{proof}

\begin{remark}\label{rem:Fibo-sym}
  Both in Lemma~\ref{lem:Fibo-rec} and in
  Proposition~\ref{prop:Fibo-rec}, we started from conditions that are
  satisfied by the pair correlation coefficients of the Fibonacci
  chain. In particular, the symmetry condition \eqref{eq:symm} was
  imposed. Interestingly, one can dispense with it as follows. If one
  considers the relations \eqref{eq:Fibo-rec} under the sole condition
  that all coefficients $\nu^{}_{\!\alpha\beta} (z)$ vanish for any point
  $z\not\in \varLambda -\varLambda$, but with no further assumption on
  the symmetry, the solution space is still one-dimensional. This can
  be proved explicitly by a slight extension of our calculations.
  Clearly, any solution can uniquely be written as the sum of a
  function that is symmetric under exchanging the indices and
  simultaneously inverting the argument with another function that is
  anti-symmetric under this operation. Since we already know a
  symmetric solution, we may conclude that the only anti-symmetric
  solution to Eq.~\eqref{eq:Fibo-rec} is the trivial one in this
  example.
\end{remark}

To proceed, let us introduce the measures
\[
    \bs{\nu}^{}_{\!\alpha \beta} \, := \,
    \sum_{z\in\vL-\vL} \nu^{}_{\alpha \beta} (z) \, \delta^{}_{z} \ts .
\]
These are translation bounded pure point measures with
$\bs{\nu}^{}_{\!\alpha \beta} (\{ z \} ) = \nu^{}_{\alpha \beta} (z)$.
Defining the invertible continuous function $f$ by $f(x) = \tau \ts x$
and its action $f\! .\ts \mu$ on a measure $\mu$ as usual by $(f\!
.\ts \mu)(g) := \mu(g\circ f)$, one finds $f\! .\ts \delta^{}_{x} =
\delta^{}_{\!  f(x)}$.  With this, a short calculation shows that one
can rewrite the recursion relations of Eq.~\eqref{eq:Fibo-rec} in
measure-valued form as
\begin{equation}\label{eq:Fibo-meas}
\begin{split}
   \bs{\nu}^{}_{\! aa} \, & = \, \tfrac{1}{\tau} \bigl(
   f\! .\ts \bs{\nu}^{}_{\! aa} + f\! .\ts \bs{\nu}^{}_{\! ab} +
   f\! .\ts \bs{\nu}^{}_{\! ba} + f\! .\ts \bs{\nu}^{}_{\! bb} \bigr) \\
   \bs{\nu}^{}_{\! ab} \, & = \, \tfrac{1}{\tau} \delta^{}_{\tau} * \bigl(
   f\! .\ts \bs{\nu}^{}_{\! aa} + f\! .\ts \bs{\nu}^{}_{\! ba} \bigr) \\
   \bs{\nu}^{}_{\! ba} \, & = \, \tfrac{1}{\tau} 
         \delta^{}_{-\tau} * \bigl(
   f\! .\ts \bs{\nu}^{}_{\! aa} + f\! .\ts \bs{\nu}^{}_{\! ab} \bigr) \\
   \bs{\nu}^{}_{\! bb} \, & = \, \tfrac{1}{\tau}
   \bigl(f\! .\ts \bs{\nu}^{}_{\! aa}  \bigr)
\end{split}
\end{equation}
where it is understood that the support of the measures on the left
hand sides is contained in $\vL-\vL$, which is uniformly discrete.
Let us note in passing that Eq.~\eqref{eq:Fibo-meas} can also be
written as a matrix convolution identity.

\begin{coro}
  Consider the renormalisation relations of Eq.~\eqref{eq:Fibo-meas},
  subject to the symmetry conditions\/
  $\widetilde{\bs{\nu}^{}_{\!\alpha \beta}} = \bs{\nu}^{}_{\!\beta
    \alpha}$ and the requirement that each measure\/
  $\bs{\nu}^{}_{\!\alpha\beta}$ is a pure point measure with support
  in\/ $\vL-\vL$. Then, there is a non-trivial solution of this system
  of equations, which is unique once the initial condition\/
  $\bs{\nu}^{}_{\!aa} \bigl( \{ 0 \} \bigr)$ is specified. The same
  conclusion holds if the support is allowed to be\/ $\oplam\bigl(
  [-\tau,\tau]\bigr)$.
\end{coro}
\begin{proof}
  Note first that, under the restriction to pure point measures, the
  condition $\widetilde{\bs{\nu}^{}_{\!\alpha \beta}} =
  \bs{\nu}^{}_{\!\beta \alpha}$ is just the reformulation of the
  symmetry relations from Eq.~\eqref{eq:symm} in this setting. The
  claim now follows from Lemma~\ref{lem:Fibo-rec} and
  Proposition~\ref{prop:Fibo-rec}.
\end{proof}

Now, observing the identity
\[
   \bs{\nu}^{}_{\! \alpha \beta} \, = \, 
   \frac{\widetilde{\delta^{}_{\!\vL_{\alpha}}} \!\circledast \ts
   {\delta^{}_{\!\vL_{\beta}}}}{\dens (\vL) } \ts ,
\]
we know from Lemma~\ref{lem:transformable} that all measures in
Eq.~\eqref{eq:Fibo-meas} are transformable, where
$\widetilde{\bs{\nu}^{}_{\!\alpha \beta}} = \bs{\nu}^{}_{\!\beta
  \alpha}$ leads to $\overline{\widehat{\bs{\nu}^{}_{\!
      \alpha\beta}}}=\widehat{\bs{\nu}^{}_{\! \beta\alpha}} $.  Since
$\widehat{f\! . \ts \mu} = \frac{1}{\tau} f^{-1}\! . \ts
\widehat{\mu}$, a Fourier transform of the relations in
Eq.~\eqref{eq:Fibo-meas} gives
\begin{equation}\label{eq:Fibo-FT}
\begin{split}
   \widehat{\bs{\nu}^{}_{\! aa}} \, & = \, \tfrac{1}{\tau^{2}}
   f^{-1}\! . \bigl( \widehat{\bs{\nu}^{}_{\! aa}} + 
   \widehat{\bs{\nu}^{}_{\! ab}} + \widehat{ \bs{\nu}^{}_{\! ba}} +  
   \widehat{\bs{\nu}^{}_{\! bb}} \bigr) \\
    \widehat{\bs{\nu}^{}_{\! ab}} \, & = \, \tfrac{1}{\tau^{2}} \,
   \ee^{-2\pi\ii \tau (.)} f^{-1}\! .\bigl(
   \widehat{\bs{\nu}^{}_{\! aa}} + \widehat{\bs{\nu}^{}_{\! ba}} \bigr) \\
    \widehat{\bs{\nu}^{}_{\! ba}} \, & = \, \tfrac{1}{\tau^{2}} \,
   \ee^{2\pi\ii \tau (.)} f^{-1}\! .\bigl(
   \widehat{\bs{\nu}^{}_{\! aa}} + \widehat{\bs{\nu}^{}_{\! ab}} \bigr) \\
    \widehat{\bs{\nu}^{}_{\! bb}} \, & = \, \tfrac{1}{\tau^{2}}
   f^{-1}\! . \widehat{ \bs{\nu}^{}_{\! bb} } \ts .
\end{split}
\end{equation}

Let us check the consistency of Eq.~\eqref{eq:Fibo-FT} with the model
set description that is available here.  From the latter, we know that
all $\widehat{\bs{\nu}^{}_{\! \alpha \beta}}$ are pure point measures
\cite{TAO}. In particular, we have $\widehat{\bs{\nu}^{}_{\! \alpha
    \beta}} = \sum_{k\in\cF} I^{}_{\alpha \beta} (k) \, \delta^{}_{k}$
with Fourier module $\cF = \ZZ[\tau]/\sqrt{5}$ and intensities
$I^{}_{\alpha \beta} (k) = \widehat{\bs{\nu}^{}_{\! \alpha \beta}} (\{
k \} )$.  An explicit calculation on the basis of
\cite[Sec.~9.4.1]{TAO}, adjusted to our use of relative frequencies,
results in the relations
\begin{equation}\label{eq:Fibo-intens}
\begin{split}
  I^{}_{aa} (k) \, & = \, \tfrac{1}{\tau^2} \sinc \bigl(\pi k^\star\bigr)^{2} \\
  I^{}_{ab} (k) \, & = \, \tfrac{1}{\tau^3} \, \ee^{-\pi\ii\tau k^\star}
   \sinc \bigl(\pi k^\star\bigr) \sinc \bigl(\tfrac{\pi k^\star}{\tau}\bigr) \\
  I^{}_{ba} (k) \, & = \, \tfrac{1}{\tau^3} \, \ee^{\pi\ii\tau k^\star}
   \sinc \bigl(\pi k^\star\bigr) \sinc \bigl(\tfrac{\pi k^\star}{\tau}\bigr) \\
  I^{}_{bb} (k) \, & = \, \tfrac{1}{\tau^{4}} 
     \sinc \bigl(\tfrac{\pi k^\star}{\tau}\bigr)^{2}
\end{split}
\end{equation}
where $I^{}_{\alpha \beta} (k) = \overline{I^{}_{\beta \alpha} (k)} =
I^{}_{\beta \alpha} (-k)$. Here, we have $I^{}_{aa} (0) =
(\nu^{}_{\nts aa} (0))^{2}$ and $I^{}_{bb} (0) = (\nu^{}_{bb}
(0))^{2}$, together with $\sum_{\alpha,\beta} I^{}_{\alpha\beta} (0) =
1$ in line with our frequency normalisation.  Note also that the
$2\!\times\! 2$-matrix $ \cI (k) := ( I^{}_{\alpha\beta}(k))$, at any
fixed $k$, is only Hermitian, not real in general. In fact, the term
`intensity' is only justified for the index pairs $aa$ and $bb$, where
one has, up to a factor $(\dens (\vL))^{2}$, the diffraction
intensities of the Dirac combs $\delta^{}_{\!\vL_{a}}$ and
$\delta^{}_{\!\vL_{b}}$, respectively. Still, the matrix $\cI (k)$ is
positive semi-definite, with $\det ( \cI (k)) = 0$ for all $k\in\RR$.

\begin{prop}\label{prop:Fibo-intens}
  For all\/ $k\in\cF$, where\/ $\cF = \ZZ[\tau]/\sqrt{5}$ is the additive
  spectrum from Eq.~\eqref{eq:Fib-spec}, the intensity functions of
  Eq.~\eqref{eq:Fibo-intens} satisfy the relations
\[
   \begin{split}
  I^{}_{aa} (k) \, & = \, \tfrac{1}{\tau^{2}} \bigl(
    I^{}_{aa} (\tau k) +  I^{}_{ab} (\tau k) + I^{}_{ba} (\tau k) 
    + I^{}_{bb} (\tau k) \bigr),\\
  I^{}_{ab} (k) \, & = \, \tfrac{1}{\tau^{2}} \, \ee^{-2\pi\ii\tau k}
    \bigl(I^{}_{aa} (\tau k) +  I^{}_{ba} (\tau k) \bigr), \\
  I^{}_{ba} (k) \, & = \, \tfrac{1}{\tau^{2}} \, \ee^{2\pi\ii\tau k}
   \bigl(I^{}_{aa} (\tau k) +  I^{}_{ab} (\tau k) \bigr),
   \quad \text{and} \\
  I^{}_{bb} (k) \, & = \, \tfrac{1}{\tau^{2}} \, I^{}_{aa} (\tau k)\ts .
\end{split}
\]
\end{prop}
\begin{proof}
  The last relation is obvious from Eq.~\eqref{eq:Fibo-intens}, as
  $(\tau k)^\star = - k^\star/\tau$ and $I^{}_{aa}(k)$ is symmetric under
  $k\mapsto -k$. For the other three, one needs some less obvious
  calculations. To do so, one has to use the fact that $k\in\cF$
  implies $k+k^\star\in\ZZ$, so that
\[
   \ee^{-2\pi\ii \tau k} \, = \, \ee^{2\pi\ii (\tau k)^\star}
   \, = \, \ee^{-2\pi\ii \frac{k^\star}{\tau}}.
\]
  The remaining explicit steps are now standard, and hence left to
  the reader.
\end{proof}

Let us mention in passing that 
Proposition~\ref{prop:Fibo-intens} also entails the scaling relation
\[
     \det (\cI (\tau k)) \, = \, \tau^{4} \, \det (\cI (k)) 
\]
for the intensity matrix.  Since the intensities are clearly bounded,
this is only compatible with $\det(\cI(k)) = 0$, as calculated earlier
from Eq.~\eqref{eq:Fibo-intens}.

Looking again at Eq.~\eqref{eq:Fibo-FT} one realises that it can
also be written in matrix form as
\begin{equation}\label{eq:Fibo-FT-matrix}
   \begin{pmatrix} \widehat{\bs{\nu}^{}_{\! aa}} \\
   \widehat{\bs{\nu}^{}_{\! ab}} \\ \widehat{\bs{\nu}^{}_{\! ba}} \\
   \widehat{\bs{\nu}^{}_{\! bb}} \end{pmatrix} \, = \,
    \frac{1}{\tau^{2}}\, \bs{A}(.) 
   \begin{pmatrix} f^{-1} \! . \ts\widehat{\bs{\nu}^{}_{\! aa}} \\
   f^{-1} \! . \ts\widehat{\bs{\nu}^{}_{\! ab}} \\ 
   f^{-1} \! . \ts \widehat{\bs{\nu}^{}_{\! ba}} \\
   f^{-1} \! . \ts \widehat{\bs{\nu}^{}_{\! bb}} \end{pmatrix} ,
\end{equation}
or $\widehat{\bs{\nu}} = \tau^{-2} \bs{A}(.) (f^{-1}\!
.\ts\widehat{\bs{\nu}})$ for short, with the matrix function
\[
    \bs{A}(k) \, = \, \begin{pmatrix} 1 & 1 & 1 & 1 \\
    \ee^{-2\pi\ii\tau k} & 0 & \ee^{-2\pi\ii\tau k} & 0 \\
     \ee^{2\pi\ii\tau k} & \ee^{2\pi\ii\tau k} & 0 & 0 \\
    1 & 0 & 0 & 0 \end{pmatrix} \, = \,  
    B(k) \otimes \overline{B(k)} \ts .
\]
Here, $\otimes$ denotes the Kronecker product (the representation of
the tensor product in the standard lexicographic choice of basis), and
$B(k)$ is the matrix function
\begin{equation}\label{eq:Fibo-family}
    B(k) \, = \, \begin{pmatrix} 1 & 1 \\
    \ee^{2\pi\ii\tau k} & 0 \end{pmatrix} \, = \,
    \begin{pmatrix} 1 & 1 \\ 0 & 0 \end{pmatrix} 
    + \ee^{2\pi\ii\tau k}  \begin{pmatrix}
    0 & 0 \\ 1 & 0 \end{pmatrix} ,
\end{equation}
where $B(0)=M$ is the substitution matrix of the Fibonacci
rule from above. Below, we will refer to $B(k)$ as the \emph{Fourier
matrix} of the inflation. The name is chosen to reflect the fact
that $B(k)$ is the matrix of phase factors that emerge from the
relative shifts of tiles within their level-$1$ supertiles.

\begin{lemma}\label{lem:B-irred}
  The matrix family\/ $\cB_{\varepsilon} := \{ B(k) \mid 0\le k <
  \varepsilon \}$ is\/ $\CC$-irreducible for any\/ $\varepsilon > 0$.
\end{lemma}
\begin{proof}
  Let $\varepsilon >0$ be fixed.  Irreducibility of
  $\cB_{\varepsilon}$ means that the only simultaneous invariant
  subspaces of the entire family are the trivial spaces, $\{ 0 \}$ and
  $\CC^{2}$. The algebra generated by $\cB_{\varepsilon}$ does not
  depend on $\varepsilon$, and equals the (complex) algebra generated
  by the two matrices
\[
    D^{}_{0} \, = \, \begin{pmatrix} 1 & 1 \\ 0 & 0
    \end{pmatrix}  \quad \text{and} \quad
    D^{}_{\tau} \, = \, \begin{pmatrix} 0 & 0 \\ 1 & 0
    \end{pmatrix} ,
\]
as is immediate from the representation in Eq.~\eqref{eq:Fibo-family}.
It is routine to check (via products and linear combinations) that
these two matrices generate the ring $\Mat (2,\CC)$ with unit group
$\GL (2,\CC)$, which is clearly irreducible as a matrix group.
\end{proof}

In fact, as we shall see in our later examples, the complex algebra
generated by the (generalised) digit matrices $D_{0}$ and $D_{\tau}$
contains important information about the inflation $\varrho$; compare
\cite{Vince} for a justification of our terminology.  In view of the
meaning of the $D$-matrices, we call this algebra the \emph{inflation
  displacement algebra} of $\varrho$, or IDA for short.  The proof of
Lemma~\ref{lem:B-irred} then also shows the following result.

\begin{coro}
  The IDA of the Fibonacci inflation rule\/ $\varrho =
  \varrho^{}_{\mathrm{F}}$ is the full matrix algebra\/ $\Mat
  (2,\CC)$, and thus irreducible. \qed
\end{coro}

Let us go back to Eq.~\eqref{eq:Fibo-FT-matrix}. Each entry in
the measure vector $\widehat{\bs{\nu}}$ has a unique decomposition
\[
   \widehat{\bs{\nu}^{}_{\! \alpha \beta}} \, = \,
   \bigl(\widehat{\bs{\nu}^{}_{\! \alpha \beta}}\bigr)_{\mathsf{pp}} + 
   \bigl(\widehat{\bs{\nu}^{}_{\! \alpha \beta}}\bigr)_{\mathsf{cont}}
\]
into its pure point (\textsf{pp}) and continuous (\textsf{cont}) part,
where the supporting sets $F_{\nts\alpha\beta}$ of the pure point parts
are (at most) countable sets, while the continuous parts are
concentrated on their complements. Defining $F = \bigcup_{\alpha\beta}
F_{\nts\alpha\beta}$, we see that $F$ is still (at most) a countable set
and that we can now write
\[ 
    \bigl(\widehat{\bs{\nu}^{}_{\! \alpha \beta}}\bigr)_{\mathsf{pp}} 
    \, = \, \widehat{\bs{\nu}^{}_{\! \alpha \beta}}\big|_{F}
    \quad \text{and} \quad
    \bigl(\widehat{\bs{\nu}^{}_{\! \alpha \beta}}\bigr)_{\mathsf{cont}} 
    \, = \, \widehat{\bs{\nu}^{}_{\! \alpha \beta}}
       \big|_{F_{\phantom{I}}^{\mathsf{c}}} 
\]
with $F^{\mathsf{c}} := \RR \setminus F$, simultaneously for all
$\alpha,\beta$. In particular, we then have a clear meaning of the
decomposition $\widehat{\bs{\nu}} =
(\widehat{\bs{\nu}})^{}_{\mathsf{pp}} +
(\widehat{\bs{\nu}})^{}_{\mathsf{cont}}$.

\begin{prop}
  Let\/ $\widehat{\bs{\nu}}$ be a solution of
  Eq.~\eqref{eq:Fibo-FT-matrix}.  Then, the pure point part\/
  $\bigl(\widehat{\bs{\nu}}\bigr)_{\mathsf{pp}}$ and the continuous
  part\/ $\bigl(\widehat{\bs{\nu}}\bigr)_{\mathsf{cont}}$ satisfy
  Eq.~\eqref{eq:Fibo-FT-matrix} separately.
\end{prop}
\begin{proof}
  The pure point and continuous parts are mutually orthogonal in the
  measure-theoretic sense; compare \cite[Prop.~8.4]{TAO}.  Since
  $A(k)$ is analytic in $k$ and $f^{-1}$ just induces a rescaling, it
  is clear that $\bigl( A(.) (f^{-1}\! . \,
  \widehat{\bs{\nu}})\bigr)_{\mathsf{pp}} = A(.) \bigl( f^{-1}\!  . \,
  (\widehat{\bs{\nu}})^{}_{\mathsf{pp}} \bigr) $, and analogously for
  the continuous parts. With the supporting set $F$ from above, 
  which we may assume to be invariant under the function $f$
  without loss of generality,  the
  evaluation of $\widehat{\bs{\nu}}$ on any Borel set $B$ can thus be
  split into two terms via $B = (B\cap F) \ts\dot{\cup} \ts(B \cap
  F^{\mathsf{c}})$, from which the claim follows by standard arguments
  because the pure point and the continuous components cannot mix.
\end{proof}

Now, the pure point and continuous parts are transformable
measures \cite{GLA}, and their (inverse) Fourier transforms
provide the unique decomposition 
\begin{equation}\label{eq:Fibo-split}
   \bs{\nu} \, = \, (\bs{\nu})^{}_{\mathsf{s}} + (\bs{\nu})^{}_{0}
\end{equation}
of $\bs{\nu}$ into its strongly almost periodic (\textsf{s}) and null
weakly almost periodic ($0$) parts, to be read componentwise as
before. Since Fourier transform is invertible on transformable
measures, these parts must then separately satisfy the renormalisation
relations of Eq.~\eqref{eq:Fibo-meas}.  Note that the parts still also
satisfy the symmetry relation, but it is not clear what the supporting
sets of the parts are. This is caused by $\vL-\vL$ \emph{not} being a
group, wherefore we initially get such a decomposition only within
$\RR$. To improve the situation, we need a smaller covering object of
$\vL - \vL$ with good harmonic properties to continue. \medskip

At this point, we do \emph{not} refer to the (known) model set
description recalled earlier, but rather proceed by employing
Strungaru's result from Section~\ref{sec:prelim}. This tells us that
both parts, $(\bs{\nu})^{}_{\mathsf{s}}$ and $(\bs{\nu})^{}_{0}$,
still are pure point measures with support in $\oplam\bigl(
[-\tau,\tau]\bigr)$. This has the following strong consequence, which
crucially builds on Proposition~\ref{prop:Fibo-rec} in a non-trivial
way.

\begin{theorem}
  Let\/ $\bs{\nu}$ be the unique solution of Eq.~\eqref{eq:Fibo-meas}
  according to Proposition~$\ref{prop:Fibo-rec}$, with initial
  condition\/ $\bs{\nu}^{}_{\! aa} \bigl( \{ 0 \} \bigr) =
  \frac{1}{\tau}$, say.  If\/ $\bs{\nu} = (\bs{\nu})^{}_{\mathsf{s}} +
  (\bs{\nu})^{}_{0}$ is the decomposition from
  Eq.~\eqref{eq:Fibo-split}, one has\/ $ (\bs{\nu})^{}_{0} =0$, which
  means that all measures\/ $\widehat{\bs{\nu}^{}_{\alpha\beta}}$ are
  pure point measures.
\end{theorem}
\begin{proof}
  From Proposition~\ref{prop:Fibo-rec}, we know that the solution of
  Eq.~\eqref{eq:Fibo-meas} is unique under the constraints formulated
  in this proposition. These contraints are met by both parts from
  Eq.~\eqref{eq:Fibo-split} separately, which are both supported on
  $\oplam\bigl( [-\tau,\tau]\bigr)$ in the CPS of Eq.~\eqref{eq:candp}
  by Strungaru's theorem \cite{Nicu}. However, we do not know how the
  initial condition is split between the two parts. However, the
  uniqueness of the solution means that the two parts are either
  proportional to one another (which is controlled by the initial
  condition at $0$), or one part is trivial.  We are clearly in the
  latter case here, because a non-zero measure cannot be strongly
  almost periodic and null-weakly almost periodic at the same time;
  compare \cite{GLA,Nicu}.
  
  Since we know that the Fourier transform of $\bs{\nu}^{}_{\! aa}$
  must have a Dirac measure with positive weight at $0$ (in fact, we
  even know that the pure point part has a relatively dense support by
  a result due to Strungaru \cite{Nicu-old}), we have
  $(\bs{\nu})^{}_{\mathsf{s}} \ne 0$, hence $(\bs{\nu})^{}_{0} =0$ and
  $\widehat{\bs{\nu}}$ is a vector of pure point measures as claimed.
\end{proof}

This provides a (partly) independent proof of the pure point nature for
the diffraction measure of the Fibonacci inflation rule. It is only
partly independent in the sense that it needs Strungaru's theorem for
Eq.~\eqref{eq:Fibo-split} as input, and this theorem still relies, to
some extent, on the CPS in the background. Our approach may be viewed
as some explicit way to prove that the Fibonacci dynamical system is
an a.e.\ 1-1 cover of its maximal equicontinuous (or Kronecker)
factor.  Note that this approach only provides the nature of the
diffraction measure, but not its explicit form, which was described
earlier via the projection formalism.

\section{Example 2: Thue--Morse and 
Rudin--Shapiro}\label{sec:TM-RS}

As another example, let us take a look at the classic sequences of
Thue--Morse and Rudin--Shapiro. Since both are examples of constant
length substitutions, the symbolic and the geometric pictures
coincide.  This results in a significant simplification in the sense
that one can derive recursion relations directly for the
autocorrelation coefficients, as explained in detail in
\cite[Chs.~10.1 and 10.2]{TAO}. Nevertheless, it is instructive to
also take a look at the proper analogues of Eq.~\eqref{eq:Fibo-rec},
which is actually in line with the general treatment in \cite{Q,Bart}.
The common feature is an additional symmetry that shows up as an
involution on the alphabet, which consists of an \emph{even} number of
letters. We call it a \emph{bar swap symmetry} and use an adjusted
alphabet to highlight its action. For other examples with a bar swap
symmetry, we refer to \cite{GTM,squiral} and references therein.

\subsection{Thue--Morse}

In the light of our general comment, we use the alphabet $\cA=\{ a,
\bar{a} \}$ and the substitution $a \mapsto a \bar{a}$, $\bar{a}
\mapsto \bar{a} a$.  The alphabet as well as the substitution rule is
invariant under the \emph{bar swap involution} $a \longleftrightarrow
\bar{a}$, which we call $P$. Let us now assume that the letters
represent intervals of length $1$, so that the coincidence of the
symbolic and the geometric picture is compatible with the embedding of
$\ZZ$ in $\RR$. Note that is suffices to study the $\ZZ$-action in
this case, as the $\RR$-action emerges from a standard suspension
\cite{EW}.

In this setting, and for each $\alpha,\beta \in \cA$, the correlation
coefficient $\nu^{}_{\!\alpha\beta}$ has support inside $\ZZ$, and the 
analogue of  Eq.~\eqref{eq:Fibo-rec} can compactly be written as
\begin{equation}\label{eq:TM-rec}
  2\, \nu^{}_{\!\alpha\beta} (z) \, = \,
  \nu^{}_{\!\alpha\beta} \bigl( \tfrac{z}{2}\bigr) +
  \nu_{\!\alpha\bar{\beta}} \bigl( \tfrac{z-1}{2}\bigr) +
  \nu^{}_{\!\bar{\alpha}\beta} \bigl( \tfrac{z+1}{2}\bigr) +
  \nu_{\!\bar{\alpha}\bar{\beta}} \bigl( \tfrac{z}{2}\bigr) \ts ,
\end{equation}
with the understanding that a coefficient always vanishes when the
argument is not an integer. It is clear that this set of four
equations is invariant under a bar swap. More precisely, applying $P$
to all first indices just permutes the four equations, and the same 
is true when $P$ is applied to all second indices, or to all indices. 

Before we analyse Eq.~\eqref{eq:TM-rec}, let us consider the
autocorrelation coefficients of the TM system. Via standard arguments,
compare \cite[Rem.~10.3]{TAO}, one needs two sets for the
decomposition of a general TM chain (resp.\ its autocorrelation) with
weights $h_{a}$ and $h_{b}$. We thus define
\begin{equation}\label{eq:TM-combi}
    \eta^{}_{\pm} (z) \, := \, 
    \bigl(\nu^{}_{\nts\nts a a} (z) +
      \nu^{}_{\nts \bar{a} \bar{a}} (z) \bigr)
   \pm \bigl(\nu^{}_{\nts\nts a \bar{a}} (z) +
      \nu^{}_{\nts \bar{a} a} (z) \bigr)
\end{equation}
and observe that the relations \eqref{eq:TM-rec} then lead to the
decoupled relations
\[
     \eta^{}_{\pm} (z) \, = \, 
       \eta^{}_{\pm} \bigl(\tfrac{z}{2}\bigr) \pm
      \tfrac{1}{2} \Bigl( \eta^{}_{\pm} \bigl(\tfrac{z-1}{2}\bigr) +
       \eta^{}_{\pm} \bigl(\tfrac{z+1}{2}\bigr) \Bigr),
\]
which is a consequence of the bar swap symmetry. 
Since $\eta^{}_{\pm}$ are functions on $\ZZ$, we now distinguish
even and odd $z$. This allows to rewrite the last recursion as
\begin{equation}\label{eq:TM-split}
\begin{split}
     \eta^{}_{\pm} (2n) & = \, \eta^{}_{\pm} (n) \ts , \\[1mm]
    \eta^{}_{\pm} (2n+1) & = \, \pm \tfrac{1}{2}
    \bigl(\eta^{}_{\pm} (n) + \eta^{}_{\pm} (n+1)\bigr),
\end{split}
\end{equation}
with $n\in\ZZ$.

It is easy to check inductively that Eq.~\eqref{eq:TM-split} implies
$\eta^{}_{+} (n) = \eta^{}_{+} (0)$ for all $n\in\ZZ$, so that
$\sum_{n\in\ZZ} \eta^{}_{+} (n) \, \delta^{}_{n} = \eta^{}_{+} (0) \,
\delta^{}_{\ZZ}$. Moreover, one finds $\eta^{}_{-} (\pm 1) = -
\frac{1}{3} \eta^{}_{-} (0)$, so that the coefficients $\eta^{}_{-}
(n)$ are the autocorrelation coefficients of the balanced Thue--Morse
chain as described in \cite[Sec.~10.1]{TAO}. Alternatively, they can
be viewed as the Fourier coefficients of the standard choice of the
maximal spectral measure in the orthocomplement of the pure point
sector, in line with the original treatment in \cite{Kaku}; see also
\cite{Q}.
\smallskip

Let us return to the original relations in Eq.~\eqref{eq:TM-rec}.
Here, the interesting situation arises that the solution space
\emph{does} depend on the support. This is an important difference to
the Fibonacci example in Section~\ref{sec:Fibo}. For a precise
formulation, we select a fixed point of the inflation, with seed $a|a$
say, and define the decomposition $\ZZ = \varLambda^{}_{a} \dot{\cup}
\varLambda^{}_{\bar{a}}$ with $\varLambda^{}_{\alpha}$ denoting all
integers that are the left endpoint of an interval of type $\alpha$.
Then, analogously to before, $S_{\!\alpha \beta} =
\varLambda_{\beta} - \varLambda_{\alpha}$ is the support of the
relative pair correlation coefficient $\nu^{}_{\!\alpha \beta}$ of the
TM system. Note that this support is the same for the entire hull,
even though we defined it via a selected fixed point, as a consequence
of strict ergodicity of the dynamical system.

\begin{prop}\label{prop:TM-dim}
  The renormalisation relations \eqref{eq:TM-rec}, viewed as relations
  between functions\/ $\nu^{}_{\!\alpha\beta} \! : \, \ZZ
  \longrightarrow \RR$ with\/ $\alpha,\beta \in\cA$, possesses a
  two-dimensional solution space.

  Moreover, under the additional restriction that\/ $\supp
  (\nu^{}_{\!\alpha\beta}) = S^{}_{\!\alpha \beta}$ for all\/
  $\alpha,\beta \in \cA$, the solution space is only one-dimensional.
\end{prop}

\begin{proof}
  Here, the finitely many self-consistency equations are those with
  arguments in $[-1,1]$. As before, they fully determine the dimension
  of the solution space, because all other equations are of a purely
  recursive nature and determine the remaining coefficients uniquely.
  Our claim is now a straight-forward calculation, which can be left
  to the reader.
\end{proof}

Let us explore the difference to the Fibonacci
inflation in more algebraic terms. Here, the Fourier matrix $B$ reads
\[
    B(k) \, = \, \begin{pmatrix}
    1 & \ee^{2 \pi \ii k} \\ \ee^{2 \pi \ii k} & 1 \end{pmatrix}
    \, = \, \one + \ee^{2 \pi \ii k} J
\]
with $J= \left(\begin{smallmatrix} 0 & 1 \\ 1 &
    0 \end{smallmatrix}\right)$.  This implies that the IDA is
generated by the commuting digit matrices
\[
     D_{0} \, = \, \one
     \quad \text{and} \quad
     D_{1} \, = \, J .
\]
As a consequence, the IDA is \emph{reducible} and certainly not the
full matrix algebra $\mathrm{Mat} (2,\CC)$. Indeed, the eigenspaces of
$J$, namely $\CC \left(\begin{smallmatrix} 1 \\ 1 \end{smallmatrix}
\right)$ and $\CC \left(\begin{smallmatrix} 1 \\ -1 \end{smallmatrix}
\right)$, are non-trivial invariant subspaces of the IDA.
Consequently, also the Kronecker product
\[
  \bs{A}(k) \, = B(k) \otimes \overline{B(k)} \, = \,
  \begin{pmatrix} 
    1 & \ee^{-2\pi\ii k} & \ee^{2\pi\ii k} & 1 \\
    \ee^{-2\pi\ii k} & 1 & 1 & \ee^{2\pi\ii k} \\
    \ee^{2\pi\ii k} & 1 & 1 & \ee^{-2\pi\ii k} \\
    1 & \ee^{2\pi\ii k} & \ee^{-2\pi\ii k} & 1 \\
  \end{pmatrix}  
\]
has a $k$-independent eigenbasis. The two eigenvectors $(1,1,1,1)^T$
and $(1,-1,-1,1)^T$ correspond to the two solutions $\eta^{}_\pm$ from
Eq.~\eqref{eq:TM-combi}. In fact, each of these eigenvectors is the
Kronecker product of one of the eigenvectors of $B(k)$ with itself.
The (symmetrised) mixed Kronecker products, one eigenvector times the
other, do not play a role here. Even though these vectors span an
invariant subspace, too, the off-diagonal sectors of the tensor
product alone cannot lead to a positive measure. This is possible only
in combination with the diagonal parts, which span invariant subspaces
already by themselves \cite{Q,Bart}.

\subsection{Rudin--Shapiro}

For this example, we use the alphabet $\cA = \{ a,b, \bar{a},
\bar{b}\}$ together with the substitution rule
\[
    \varrho^{}_{\mathrm{RS}} : \quad
    a \mapsto ab \, , \quad 
    b \mapsto a\bar{b} \, , \quad
   \bar{a} \mapsto \bar{a}\bar{b} \, , \quad 
   \bar{b} \mapsto \bar{a} b \ts .
\]
This is equivalent to the formulation used in \cite[Sec.~4.7.1]{TAO}
via the identifications $a\, \widehat{=}\, 0$, $\bar{a}\,\widehat{=}\,
3$, $b \, \widehat{=}\, 2$ and $\bar{b}\,\widehat{=}\, 1$. As before,
both the alphabet and the substitution rule are invariant under the
complete bar swap $a \longleftrightarrow \bar{a}$,
$b\longleftrightarrow \bar{b}$, again called $P$. So, we have
\[
    P (\cA) \, = \, \cA  \quad \text{and} \quad
    \varrho^{}_{\mathrm{RS}} \circ P \, = \,
     P \circ \varrho^{}_{\mathrm{RS}} \ts ,
\]
which has similar consequences as in the previous example. In fact,
we have a little more in this example: One can also check that
\[
    R \circ \varrho^{}_{\mathrm{RS}}  \, = \,
    E \circ \varrho^{}_{\mathrm{RS}} \circ E \ts ,
\]
where $R$ is the permutation $a\mapsto b \mapsto \bar{a} \mapsto
\bar{b}\mapsto a$ of order $4$, with $R^2=P$, and $E$ is the letter
exchange involution defined by $a \longleftrightarrow b$, $ \bar{a}
\longleftrightarrow \bar{b}$. Since $R\circ E$ is another involution,
the group generated by $R$ and $E$ is the dihedral group of order $8$.

The Fourier matrix for $\varrho^{}_{\mathrm{RS}}$  reads
\[
    B(k)  \, = \, D_{0} + \ee^{2 \pi \ii k} D_{1} \ts ,
\]
with the (non-commuting) generating digit matrices
\[
     D_{0} \, = \, \begin{pmatrix} 
     1 & 1 & 0 & 0 \\ 0 & 0 & 0 & 0\\
     0 & 0 & 1 & 1 \\ 0 & 0 & 0 & 0 \end{pmatrix}
     \quad \text{and} \quad
     D_{1} \, = \, \begin{pmatrix}
     0 & 0 & 0 & 0 \\ 1 & 0 & 0 & 1 \\
     0 & 0 & 0 & 0 \\ 0 & 1 & 1 & 0 \end{pmatrix} .
\]
As one can easily check, $\mathrm{ker} (D_{0}) \cap \mathrm{ker}
(D_{1}) = \CC \ts (1,-1,1,-1)^{t}$, which is a non-trivial invariant
subspace. The IDA generated by $D_{0}$ and $D_{1}$ is thus reducible,
but the two digit matrices cannot be diagonalised
simultaneously. However, one has the block diagonal form
\[
   T_{P} D_{0} T_{P} \, = \,  \begin{pmatrix}
     1 & 1 & 0 & 0 \\ 0 & 0 & 0 & 0\\
     0 & 0 & 1 & 1 \\ 0 & 0 & 0 & 0 \end{pmatrix}
     \quad \text{and} \quad
    T_{P} D_{1} T_{P} \, = \, \begin{pmatrix}
     0 & 0 & 0 & 0 \\ 1 & 1 & 0 & 0 \\
     0 & 0 & 0 & 0 \\ 0 & 0 & 1 & -1 \end{pmatrix} ,
\]
where $T_{P} = \frac{1}{\sqrt{2}}\, \left( \begin{smallmatrix}
1 & 1 \\ 1 & -1 \end{smallmatrix}\right) \otimes \one^{}_{2}$
is an involution that is induced by the bar swap map $P$.
\smallskip

The renormalisation relations for the pair correlation coefficients
read
\begin{equation}\label{eq:RS-rec}
\begin{split}
  2\, \nu^{}_{\nts aa} (z) \, & = \,
    \nu^{}_{\nts aa} \bigl( \tfrac{z}{2}\bigr) +
    \nu^{}_{\nts ab} \bigl( \tfrac{z}{2}\bigr) +
    \nu^{}_{ba} \bigl( \tfrac{z}{2}\bigr) +
    \nu^{}_{bb} \bigl( \tfrac{z}{2}\bigr) \ts , \\
  2\,\ts \nu^{}_{\nts ab} (z) \, & = \,
   \nu^{}_{\nts aa} \bigl( \tfrac{z-1}{2}\bigr) +
    \nu^{}_{\nts a\bar{b}} \bigl( \tfrac{z-1}{2}\bigr) +
    \nu^{}_{ba} \bigl( \tfrac{z-1}{2}\bigr) +
    \nu^{}_{b\bar{b}} \bigl( \tfrac{z-1}{2}\bigr) \ts , \\
  2\, \nu^{}_{ba} (z) \, & = \,
    \nu^{}_{\nts aa} \bigl( \tfrac{z+1}{2}\bigr) +
    \nu^{}_{\nts ab} \bigl( \tfrac{z+1}{2}\bigr) +
    \nu^{}_{\bar{b}a} \bigl( \tfrac{z+1}{2}\bigr) +
    \nu^{}_{\bar{b}b} \bigl( \tfrac{z+1}{2}\bigr) \ts , \\
  2\,\ts \nu^{}_{bb} (z) \, & = \,
    \nu^{}_{\nts aa} \bigl( \tfrac{z}{2}\bigr) +
    \nu^{}_{\nts a\bar{b}} \bigl( \tfrac{z}{2}\bigr) +
    \nu^{}_{\bar{b}a} \bigl( \tfrac{z}{2}\bigr) +
    \nu^{}_{\bar{b}\bar{b}} \bigl( \tfrac{z}{2}\bigr) \ts .
\end{split}
\end{equation}
Here, each line actually represents four equations, the other three
being obtained by simultaneously applying $P$ to all first, to all
second, or to all indices at once. Let $S^{}_{\!\alpha \beta}$ again
denote the supports of the coefficient functions, defined in complete
analogy to above. Then, with essentially the same arguments as in
Proposition~\ref{prop:TM-dim}, one finds the following result.

\begin{prop}\label{prop:RS-dim}
  The\/ $16$ renormalisation relations specified by
  Eq.~\eqref{eq:RS-rec}, viewed as relations between functions\/
  $\nu^{}_{\!\alpha\beta} \! : \, \ZZ \longrightarrow \RR$ with\/
  $\alpha,\beta \in\cA$, possesses a two-dimensional solution space.

  Moreover, under the additional restriction that\/ $\supp
  (\nu^{}_{\!\alpha\beta}) = S^{}_{\!\alpha \beta}$ for all\/
  $\alpha,\beta \in \cA$, the solution space is only one-dimensional.
  \qed
\end{prop}

The next step, as in the previous section, consists in the analysis
of the Kronecker product
\[
   \bs{A}(k) \, = \, B(k) \otimes \overline{B(k)} \, = \,
   D_{0} \otimes D_{0} + D_{1} \otimes D_{1} +
   \ee^{2 \pi \ii k} D_{1} \otimes D_{0} +
   \ee^{-2 \pi \ii k} D_{0} \otimes D_{1} 
\]
with respect to the splitting into even and odd sectors under $x
\otimes y \mapsto \overline{y\otimes x}$, viewed as \emph{real} vector
spaces, and with respect to further invariant subspaces, which exist
as a result of Proposition~\ref{prop:RS-dim}.  As we have seen above,
with the change of basis induced by $T_P$, $B(k)$ can be brought to
block-diagonal form, $B(k)=B_+(k)\oplus B_-(k)$, which acts on a 
direct sum $V_+\oplus V_-$, both of whose summands are separately 
left invariant. As in the Thue--Morse case, only the tensors in
$V_+\otimes V_+$ and $V_-\otimes V_-$ can lead to positive measures. 
The off-diagonal parts in the tensor product \emph{alone} cannot lead 
to a positive measure. As $B_+(k)$ has rank one, the corresponding 
Kronecker product $B_+(k)\otimes\overline{B_+(k)}$ also has rank one, 
and is symmetric. On the subspace $V_+\otimes V_+$, $A(k)$ acts
with an eigenvalue $2+2\cos(2\pi k)$, which results in the pure point
part of the spectrum of the Rudin--Shapiro sequence, just as in the
Thue--Morse case. On the subspace $V_-$, on the other hand, $B_-(k)$ 
generates the full matrix algebra, and its tensor product algebara, 
generated by $B_-(k)\otimes\overline{B_-(k)}$, is irreducible, too, 
if confined to the symmetric sector of $V_-\otimes V_-$. This 
irreducible invariant sector carries the other spectral component, 
which is known to be absolutely continuous; compare \cite{Q,Bart}.

\section{Example 3: Twisted silver mean}\label{sec:sm-mixed}

The Thue--Morse and Rudin--Shapiro sequences from the last section are
two examples of structures with a mixed spectrum. Many more, also
higher-dimensional ones, can be found in
\cite{NF1,NF2,TAO,GTM,BGG,squiral,Bart}.  Most of them have two
features in common, which they share with the Thue--Morse and
Rudin--Shapiro sequences: They are based on constant-length
substitutions, where the letters (which can be identified with the
prototiles here) come in pairs, such that prototiles within a pair are
geometrically equal, but still differ in their type. We denote this,
as before, by the presence or absence of a \emph{bar}.

To expand on this, let $\cA$ denote the set of prototiles. The key
feature then is that the map $P \! : \, \cA \longrightarrow \cA$ which
changes (or swaps) the bar status of all prototiles simultaneously, so
$\alpha \mapsto \bar{\alpha}$ for all $\alpha\in\cA$ with
$\bar{\bar{\alpha}} = \alpha$, commutes with the inflation rule, and
is thus a symmetry. Clearly, the map $P$ is an involution and has an
obvious extension to arbitrary finite patches of tiles, and then also
to the hull defined by a primitive inflation rule for $\cA$. By slight
abuse of notation, we always denote the bar swap by $P$.  The map $P$
is the key to constructing examples with mixed spectrum, and we shall
see in this section that it is \emph{not} confined to constant length
substitutions.

\subsection{General Setup}

Before constructing particular examples, let us analyse the
general situation. For ease of exposition, we assume a 
one-dimensional tiling, though the entire construction works in
higher dimensions as well. Let $\sigma$ be a primitive inflation rule on
a prototile set $\cA$ with bar swap symmetry, so that 
\[
    P(\cA) \, = \, \cA  \quad \text{and} \quad
     \sigma \circ P \, = \, P \circ \sigma
\]
hold for the map $P$ defined above. Let $\Omega$ be the hull that is
generated by $\sigma$; see \cite{TAO} for background.  We will now
construct a globally 2-1 factor map $\varphi$ from the dynamical
system $(\Omega,\RR)$ to a factor dynamical system $(\Omega',\RR)$,
which identifies tilings related by a bar swap, so $\varphi
(\omega) = \varphi \bigl( P (\omega))$ for all $\omega \in \Omega$.
By construction, the map $\varphi$ commutes with the translation
action, so that we may suppress the latter without danger of confusion.

To construct such a factor map explicitly, we first rewrite the
inflation $\sigma$ in terms of \emph{collared tiles}. The latter
consist of pairs of tiles $t_1t_2$, where $t_1$ is the actual tile
and $t_2$ is a (one-sided) collar of $t_1$. This is seen as a label
attached to $t_1$ that specifies the type of the right neigbour tile
of $t_1$.  Obviously, there is an induced inflation rule
$\widetilde\sigma$ on the collared tiles. If $\sigma(t_1)=u_1\dots
u_k$, and $\sigma(t_1t_2)=u_1\dots u_\ell$, the inflation of the
collared tile $t_1t_2$ consists of the sequence of collared tiles
$u_iu_{i+1}$, where $i$ runs from $1$ to $k$. The new inflation is
still primitive, and defines a unique hull, which we call
$\widetilde{\Omega}$.  Clearly, collaring is a local operation, which
has a local inverse (the forgetful map wiping out the collars), and it
commutes with the inflation. In other words, the collared and the
uncollared inflations, $\widetilde{\sigma}$ and $\sigma$, define
mutually locally derivable (MLD) hulls, which are thus topologically
conjugate; compare \cite[Sec.~5.2]{TAO}.

We now observe that also $\widetilde\sigma$ has a bar swap symmetry,
which simulaneously swaps the bar status of the tile and its
collar. We can now consider the factor map $\varphi$ that is induced
by identifying collared tiles related by a bar swap. Clearly,
$\varphi$ also identifies pairs of global tilings related by a bar
swap, so $\varphi (\widetilde{\omega}) = \varphi\bigl( P
(\widetilde{\omega})\bigr)$ for all $\widetilde{\omega}\in
\widetilde{\Omega}$. Since $\Omega \simeq \widetilde{\Omega}$, our
procedure induces a unique mapping from $\Omega$ to $\Omega' :=
\varphi(\widetilde{\Omega})$, which we simply call $\varphi$ again.

Let $\omega' \in \Omega'$ now be a tiling in the image of $\varphi$,
given by a bi-infinite sequence of tiles $t_i$, where $i\in\ZZ$. The
collared tile $t_0t_1$ has exactly two possible preimages --- let us
chose one of them. The neighbouring collared tile, $t_1t_2$, also has
two preimages, but as we have already chosen a preimage of the
collared tile $t_0t_1$, and thus of $t_1$, there is only one choice
left. Continuing like this, we see that, once we have chosen a preimage
of the collared tile $t_0t_1$, the lifts of all other tiles to the
right is fixed, and analogously to the left as well. Consequently,
$\omega'$ has precisely two preimages, and the mapping
$\varphi \! : \, \Omega \longrightarrow \Omega'$ is globally 2-1.

Wiping out all bars from the tiles of the original inflation
$\sigma$ also induces a factor map, but this one need not be globally
2-1. Let $\Omega''$ denote the image under this map and note that
the latter must also be a factor of $\Omega'$, so we actually
have a sequence of factor maps
\[
\begin{CD}
     \Omega  @>\varphi>\text{2-1}> \Omega' 
     @>\varphi'>\text{1-1 a.e.}> \Omega'',
\end{CD}
\]
where $\Omega'$ is obtained by identifying \emph{collared} tiles that
are related by a bar swap, while $\Omega''$ is obtained by identifying
\emph{original} tiles related by a bar swap. The second map $\varphi'$
is 1-1 a.e., because the composition of $\varphi$ and $\varphi'$ is
a.e.\ 2-1.  Almost all tilings in $\Omega''$ consist of a single, infinite 
order supertile, and these have exactly two preimages in $\Omega$,
which differ by a bar swap. Only tilings consisting of two adjacent
infinite order supertiles may have more than two preimages, but these
are of measure zero. $\Omega'$ and $\Omega''$ may coincide (in the
Thue--Morse case they to), but in general they are different. Note
that all systems under consideration here are strictly ergodic.

Each of the translation dynamical systems has a \emph{maximal
  equicontinuous factor} (MEF), also known as their Kronecker factor,
so that the above sequence of factor maps can be completed as follows,
\begin{equation}\label{eq:facmaps}
\begin{CD}
   \Omega  @>\varphi>\text{2-1}> \Omega' 
   @>\varphi'>\text{1-1 a.e.}> \Omega'' \\
   @VV{\xi}V @VV{\xi'}V @VV{\xi''}V \\
   \Omega^{}_\text{MEF}  @>\psi>> \Omega_\text{MEF}' 
         @= \Omega_\text{MEF}''
\end{CD}
\end{equation}
It is well known (compare \cite{Barge}) that a 
tiling dynamical system has pure point
dynamical spectrum if and only if the factor map to the underlying MEF
is 1-1 almost everywhere. Let us now assume that $\Omega'$, and thus
also $\Omega''$ by standard results \cite{BL-2}, has pure point
spectrum, so that both $\xi'$ and $\xi''$ are 1-1 a.e. As $\varphi'$
is 1-1 a.e., the MEFs $\Omega_\text{MEF}'$ and $\Omega_\text{MEF}''$
are equal.  There still remain two possibilities, however. The map
$\psi$ is a group homomorphism, and can thus be either 1-1 or 2-1. If
$\psi$ is 1-1, this implies that $\xi$ is 2-1 a.e., so that $\Omega$
must have mixed spectrum, whereas, if $\psi$ is 2-1, $\xi$ is 1-1
a.e., and $\Omega$ has pure point spectrum.

In order to compare $\Omega^{}_\text{MEF}$ and $\Omega_\text{MEF}'$,
we need to compare their respective modules of eventual return
vectors.  Recall that $r$ is a \emph{return vector} of a tiling
dynamical system $\Omega$ when there exist two tiles $t_1$ and $t_2$
of the same type in some tiling $\omega\in\Omega$ such that the
distance of their left endpoints is $r$.  The module of return
vectors, $\mathcal{R}_\Omega$, is the $\ZZ$-span of all return
vectors.  This is a finitely generated submodule of the $\ZZ$-module
generated by all tile lengths, $\mathcal{T}^{}_{\Omega}$. The (additive) pure
point spectrum of $\Omega$ now consists of all those $k\in \RR$ such
that, for any return vector $r$, one has $e^{2\pi \ii \lambda^nk
  r}\longrightarrow 1$ as $n\to\infty$ by \cite{Sol}. This pure 
point spectrum can only be non-trivial if the inflation factor 
(or multiplier) $\lambda$ is a PV number, which we assume.

The quantity of interest now is the $\ZZ$-module of \emph{eventual
return vectors}, given by 
\[
   \mathcal{M}^{}_{\Omega} \, = \,
   \big\langle\{x\in\mathcal{T}^{}_{\Omega} \mid \lambda^n x \in 
    \mathcal{R}^{}_\Omega, \text{ for some } n\in\NN\}
    \big\rangle_{\ZZ} \ts .
\]
Clearly, $\Omega$ and $\Omega'$ have the same MEF if and only if
$\mathcal{M}^{}_{\Omega} = \mathcal{M}^{}_{\Omega'}$. For constant
length substitutions, the module $\mathcal{M}^{}_{\Omega}$ is known as
the height lattice; compare \cite{NF2}. If $\mathcal{M}^{}_{\Omega}=h
\ts \ZZ$, with $|h|>1$, the substitution is said to have non-trivial
height $\lvert h \rvert$.  What matters in our context, however, is not whether any
height is non-trivial, but whether $\Omega$ and $\Omega'$ have the
same height.

\subsection{Extending the silver mean inflation}

Let us now look at concrete examples. Our goal is to construct almost
2-1 extensions of the silver mean inflation \cite[Ex.~4.5]{TAO}
\[
   \sigma^{}_{\nts\mathsf{sm}} : \quad a \mapsto aba \, , 
     \quad b \mapsto a \ts ,
\]
in such a way that we gain a bar swap symmetry. The scaling factor (or
inflation multiplier) of $\sigma^{}_{\nts\mathsf{sm}}$ is
$\lambda=1+\sqrt{2}$, and the natural tile lengths for $a$ (long) and
$b$ (short) are $\lambda$ and $1$, respectively. Here, $\lambda$ is a
Pisot unit, and it is well known that $\sigma$ generates tilings with
pure point spectrum; see \cite[Chs.~7 and 9]{TAO} for details.

We now add a barred version of each tile, thus giving the new
prototile set $\cA = \{ a, b, \bar{a}, \bar{b}\}$, and introduce an
inflation which is primitive and commutes with the bar swap
invlution $P$. Furthermore, this inflation shall reduce to $\sigma$
under the identification of $a$ with $\bar{a}$ and $b$ with
$\bar{b}$. A first attempt of such an inflation could be
\[
   \widetilde{\sigma} : \quad
   a \mapsto a \bar{b} a \, , \quad 
   b \mapsto a \, , \quad
   \bar{a} \mapsto \bar{a} b \bar{a} \, , \quad 
   \bar{b} \mapsto \bar{a} \ts .
\]
It is easy to see that, for any element of the hull $\Omega$ defined
by $\widetilde\sigma$, exactly every second tile carries a bar. As
multiplication by $\lambda$ induces a isomorhism on
$\mathcal{T}^{}_{\Omega}$, this means that the module of eventual
return vectors $\mathcal{M}^{}_{\Omega}$ is an index-2 submodule of
$\mathcal{T}^{}_{\Omega}$, whereas $\mathcal{M}^{}_{\Omega'} =
\mathcal{T}^{}_{\Omega}$. The multiplicity of the map $\psi$ in
diagram \eqref{eq:facmaps} is thus $2$, and the map $\xi$ is 1-1 a.e.\
in this case, so that $\Omega'$ still has pure point spectrum.

Our next attempt is the inflation $\bar\sigma$, given by
\[
    \bar\sigma : \quad
   a \rightarrow ab\bar{a} \, , \quad 
   b \rightarrow \bar{a} \, , \quad
   \bar{a} \rightarrow \bar{a}\bar{b}a \, , \quad 
   \bar{b} \rightarrow a \ts .
\]
This inflation is primitive, commutes with $P$, and it is easy to see 
that $\mathcal{R}_\Omega=\mathcal{M}_\Omega=\mathcal{T}_\Omega$, so that
$\bar\sigma$ must have mixed spectrum.  We call it the \emph{twisted
silver mean} (TSM) inflation. It is instructive to have a
closer look at $\Omega'$. For that purpose, we rewrite the
inflation in terms of collared tiles:
\[
   \sigma^{}_{1} : \quad
   A\mapsto CD\bar{B} \, , \quad 
   B\mapsto CD\bar{A} \, , \quad 
   C\mapsto CD\bar{A} \, , \quad
   D \mapsto \bar{A}  \, ,
\] 
together with $\bar{\alpha} \mapsto P (\sigma^{}_{1} (\alpha))$ for
all $\alpha \in \{ A, B, C, D \}$, where $A=aa$, $B=a\bar{a}$,
$C=ab$, and $D=b\bar{a}$. There are three variants of the
long tile, and one short tile. However, since $\sigma^{}_{1} (B) =
\sigma^{}_{1} (C)$, we can actually merge the two prototiles $B$ and
$C$, and then rename the old $D$ as the new $C$, so that we arrive at
\begin{equation}\label{eq:barsubst}
    \sigma^{}_{2} : \quad
    A \mapsto BC\bar{B} \, , \quad 
    B \mapsto BC\bar{A} \, , \quad 
    C \mapsto \bar{A} \ts , \quad 
\end{equation}
once again with $\bar{\alpha} \mapsto P (\sigma^{}_{2} (\alpha))$.
The hulls of $\sigma_1$ and $\sigma_2$ are MLD, so that we can stick
to the latter. If we wipe out the bars in $\sigma_2$, we obtain the
inflation for $\Omega'$, which has pure point spectrum in this case,
and actually is a model set by standard arguments \cite{BLM}, because
$\Omega'$ is an a.e.\ 1-1 extension of the original silver mean hull.

For each prototile type, when combining a tile with its barred
version, there is an associated \emph{window} in the CPS
\begin{equation}\label{eq:candptm}
\renewcommand{\arraystretch}{1.2}\begin{array}{r@{}ccccc@{}l}
   & \RR & \xleftarrow{\,\;\;\pi\;\;\,} & \RR \times \RR & 
        \xrightarrow{\;\pi^{}_{\mathrm{int}\;}} & \RR & \\
  \raisebox{1pt}{\text{\footnotesize dense}\,} \hspace*{-1ex}
   & \cup & & \cup & & \cup & \hspace*{-1ex} 
   \raisebox{1pt}{\,\text{\footnotesize dense}} \\
   & \ZZ[\sqrt{2}\,] & \xleftarrow{\; 1-1 \;} & \cL & 
        \xrightarrow{\; 1-1 \;} &\ZZ[\sqrt{2}\,] & \\
   & \| & & & & \| & \\
   & L & \multicolumn{3}{c}{\xrightarrow{\qquad\qquad\;\,\star
       \,\;\qquad\qquad}} 
       &  {L_{}}^{\star\nts}  & \\
\end{array}\renewcommand{\arraystretch}{1}
\end{equation}
When we distinguish according to the bars, however, we do not have a
model set. Nevertheless, for any fixed tiling $\omega$, we can still
take the set of left endpoints of all tiles of a given type (which is
a subset of $\ZZ[\lambda] = \ZZ[\sqrt{2}\,]$ in our setting), and see
what the closure of the image under the $\star$-map of this set is. In
this way, we can determine a \emph{covering window} for each tile
type. Of course, we cannot expect these covering windows for the
different tile types to be disjoint. The result is shown in the
Fig.~\ref{fig:win}.

\begin{figure}[ht]
\begin{center}
 \includegraphics[width=0.7\textwidth]{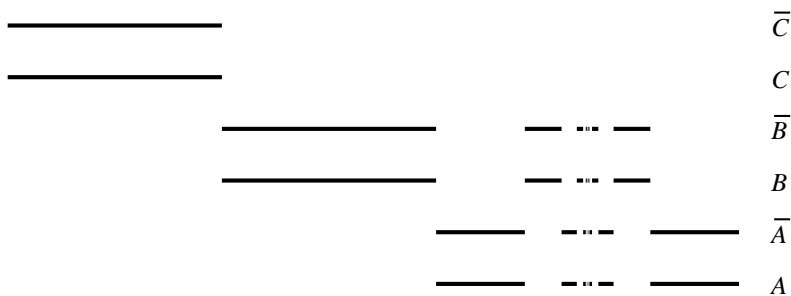}
 \end{center}
 \caption{\label{fig:win} Covering windows of the inflation 
  $\sigma_2$.}
\end{figure}

We see that a barred and an unbarred tile always share the same
covering window. Of course, every left endpoint of a tile is either
the endpoint of a barred or an unbarred tile, but not both. So, the
identical covering windows emerge as the closure of disjoint point
sets. This means that one cannot determine the bar status of a tile by
looking at its internal space coordinate. The bars represent a sort of
\emph{chemical modulation} of the original silver mean tiling, where
the modulation is completely independent of internal space
coordinates.

Another interesting point is that the hull $\Omega''$ of the original
silver mean tiling is indeed different from $\Omega'$ as generated by
the inflation \eqref{eq:barsubst}. The latter has two long tiles, not
one, and is \emph{not} MLD to $\Omega''$.  To see this, let us look at
the structure of the windows. In suitable units, the total window is
an interval of length $1+\lambda=2+\sqrt{2}$, which is split into
three pieces of length $1$, $1$, and $\sqrt{2}$. The first bit belongs
to the short tile, the second to one of the long tiles, and the third
part is fractally split between the two long tiles.  It is first
assigned to one of them, but is then split in the proportion
$1:\sqrt{2}:1$, and the middle part is assigned to the other long
tile, where it is again split into three pieces, the middle one of
which is assigned to the first long tile, and so on. Although all
interval boundaries showing up in this process are contained in
$\ZZ[\lambda^{\star}] = \ZZ[\sqrt{2}\,]$, there is an accumulation
point of interval boundaries in the middle of the subwindow of length
$\sqrt{2}$, which is \emph{not} contained in $\ZZ[\sqrt{2}\,]$. This
point represents an extra singular cut in the CPS. Two tilings
correspond to this point, both of which project to the same silver
mean tiling.

This degeneracy also shows up in the cohomology.  Using the
Anderson--Putnam method \cite{AP}, it is routine to compute
$H^1(\Omega)=\ZZ^3$ and $H^1(\Omega')=\ZZ^2$, which is in line with
\cite[Thm.~5.1]{GHK}: The extra dimension comes from the extra
$\ZZ[\sqrt{2}\,]$-orbit of singular points.

\subsection{Spectral Structure}

From the last section, we know that the TSM tiling dynamical system
has a mixed spectrum. So, in addition to the well-known pure point
part, there must be a continuous part in the spectrum. Here, we want
to investigate what the nature of this continuous part is.

We have seen that the dynamical system of the twisted silver mean
tiling is an almost 2-1 extension of the dynamical system of the
original silver mean tiling. Due to the bar swap symmetry, the Hilbert
space of square-integrable functions on the hull,
$L^2(\Omega_\mathsf{tsm})$, is a tensor product, and can be split into
an even and an odd sector under the bar swap symmetry:
\begin{equation}
  L^2(\Omega_\mathsf{tsm}) \, = \, 
  L^2(\Omega_\mathsf{sm}) \otimes \CC^2 \, = \,
  \bigl(L^2(\Omega_\mathsf{sm}) \otimes \chi^{}_{+} \bigr) \oplus  
  \bigl(L^2(\Omega_\mathsf{sm}) \otimes \chi^{}_{-}\bigr)
  =: \, \mathcal{H}_{+} \oplus \mathcal{H}_{-} \ts , 
\label{hilbertspace}
\end{equation}
where $\chi^{}_{\pm}$ is the even/odd character of the bar swap $P$.
The factor map induces a corresponding map on the Hilbert spaces.  It
sends the first summand isomorphically to $L^2(\Omega_\mathsf{sm})$,
and has the second summand as its kernel. The translation group acts
on this Hilbert space via a unitary representation, and leaves both
sectors separately invariant.

\begin{lemma}\label{lem:specpure}
  The spectral measure of the translation action on\/
  $\Omega^{}_{\mathsf{tsm}}$, confined to either of the two sectors\/
  $\mathcal{H}_{+}$ and\/ $\mathcal{H}_{-}$, is spectrally pure, so it
  has only one non-vanishing component in its Lebesgue
  decomposition. In particular, it is a pure point measure on\/
  $\mathcal{H}_{+}$, and a continuous one on\/ $\mathcal{H}_{-}$,
  where the latter is either purely singular continuous or purely
  absolutely continuous.
\end{lemma}
\begin{proof}
  This is a consequence of known results on index-$2$ extensions of
  irrational rotations; compare \cite{Q,Hel}. On the symbolic side,
  our system is almost everywhere $1:1$ over the irrational rotation
  defined by the silver number. There, the claim follows from
  \cite[Cor.~3.6]{Q}; see also \cite{Hel}. Since this purity law is a
  Hilbert space result, it carries over to the shift space.  Moreover,
  the same result remains true if one goes to the continuous
  translation action of $\RR$ as obtained by a standard suspension;
  see \cite{CFS,EW} for background.

  Our geometric setting can be viewed as a suspension with two
  heights. Since the inflation multiplier is a PV unit, the
  symbolic and the geometric dynamical systems are conjugate
  by \cite[Thm.~3.1]{CS}. The two spectral types are thus the same,
  and our claim follows.
\end{proof}

Lemma~\ref{lem:specpure} still leaves two possibilities for the
spectral type in the odd sector. In order to discriminate between the
two, we look at the asymptotic behaviour of the correlation functions
$\nu_{\alpha\beta}$, in the same way as has been done for the
Thue--Morse substitution. Because of the purity result of
Lemma~\ref{lem:specpure}, we can in fact take any combination of
correlation functions which is odd under the bar swap. A simple such
combination is the auto-correlation of a structure where the left end
points of the unbarred tiles are decorated with a weight $+1$, and
those of the barred tiles with a weight $-1$. So, let $\varLambda$ be
the set of all left endpoints of tiles of a tiling $\omega$, and set
$w(x)=\pm1$, if $x\in\varLambda$ and $x$ is the left endpoint of an
unbarred (barred) tile. With $\varLambda_R=\varLambda\cap(-R,R)$, the
relevant autocorrelation coefficients then are
\begin{equation}
   \nu^{}_\mathsf{tsm}(z)\, = \lim_{R\to\infty}\frac1{|\varLambda_R|}
   \sum_{\substack{x,y\in\varLambda_R \\ x-y=z}}w(x) \ts w(y).
\label{eta-tsm}
\end{equation}
Note that we have used here the same relative normalisation as for the
correlation functions of Eq.~\eqref{eq:Fibo-rec}. In practice, instead
of computing the limit \eqref{eta-tsm}, it is more convenient to
partition $\varLambda$ into certain patches, which all have
well-defined relative frequencies, and add up the contributions of
these patches, weighted by their frequencies.

We now have to determine the decay or non-decay of
$\nu_\mathsf{tsm}(z)$ as $z\to\infty$. 

\begin{lemma} 
  One has $\,\lim_{n\to\infty}\nu^{}_\mathsf{tsm}(z_n) = 1-\sqrt{2}$,
  where $z_n = (1+\lambda)\ts \lambda^n$.
\label{nu-non-decay}
\end{lemma}

\begin{proof}
  Before we do any concrete computations, let us sketch the strategy
  of the proof.  The contributions to $\nu^{}_\mathsf{tsm}(1+\lambda)$
  come from pairs of tiles where the left endpoint of the second tile
  $t_2$ is located, by a shift of $1+\lambda$, to the right of that of
  tile $t_1$. For any such pair, $t_2$ is the second neighbour to the
  right of $t_1$. We can therefore determine all possible triples of
  three consecutive tiles in the tiling, and add up their
  contributions to the correlation at distance $1+\lambda$, weighted
  with the relative frequency of each triple. The relative frequencies
  can be determined by Perron--Frobenius theory \cite{Q,TAO}. We
  regard the triples as doubly right-collared tiles, and determine the
  inflation matrix on the induced inflation on these collared
  tiles. The relative frequencies are then given by the components of
  the right Perron--Frobenius eigenvector of that matrix.

  For the correlation at distance $(1+\lambda)\ts \lambda^n$, we do
  the same with triples of supertiles of order $n$. These have the
  same relative frequencies as the triples of tiles. The pairs of
  tiles contributing to the correlation then consist of a left tile in
  the left supertile of the triple, and a corresponding right tile at
  distance $(1+\lambda)\ts \lambda^n$, which is contained in one of
  the other two (collar) supertiles. Since the underlying silver mean
  tiling has pure point spectrum, the density of tiles $t_1$
  \emph{not} having a corresponding tile $t_2$ with the same geometry
  at distance $(1+\lambda)\ts \lambda^n$ asymptotically vanishes as
  $n\to\infty$.  This will become evident from the explicit
  computations below.  

  To this end, we have to look for pairs of tiles which either
  \emph{match} (have the same geometry and bar status) or
  \emph{anti-match} (have the same geometry, but opposite bar
  status). Specifically, if $p_1$ and $p_2$ are two patches of tiles
  with the same support, we denote by $n_\pm(p_1,p_2)$ the number of
  tiles in $p_1$ that have a (anti-)matching partner in $p_2$ at the
  same position. We are then interested in the \emph{asymptotic
    overlap}
\begin{equation}\label{eq:asympolap}
     c(p_1,p_2) \, = \lim_{n\to\infty} \tilde{c}
     \bigl(\bar\sigma^n(p_1),\bar\sigma^n(p_2)\bigr),
\end{equation}
where
\[
    \tilde{c}(p_1,p_2)\, = \,
     \frac{n_+(p_1,p_2) - n_-(p_1,p_2)}
            {n_+(p_1,p_2) + n_-(p_1,p_2)} \ts .
\]
The possible triples of tiles are given in Table~\ref{tab:ttrip},
along with their first inflations and relative frequencies. The
subpatches that need to be compared are underlined. 

\begin{table}[t]
{\renewcommand{\arraystretch}{1.3}
\begin{tabular}{|c|l|l|l|}
\hline
no. & triple & inflated triple & rel. frequency \\
\hline
1 & $aaa$ & $\underline{ab\bar{a}}\ 
   a\underline{b\bar{a}\ a}b\bar{a}$ & $x=-1+\frac{3\sqrt{2}}{4}$  \\
2 & $a\bar{a}\bar{b}$ & $\underline{ab\bar{a}}\ 
   \bar{a}\underline{\bar{b}a\ a}$ & $x=-1+\frac{3\sqrt{2}}{4}$ \\
3 & $ab\bar{a}$ & $\underline{ab\bar{a}}\ \bar{a}\ 
   \underline{\bar{a}\bar{b}a}$ & $y=\frac{1}{2}-\frac{\sqrt{2}}{4}$ \\
4 & $aab$ & $\underline{ab\bar{a}}\ a\underline{b\bar{a}\ 
   \bar{a}}$ & $z=\frac{3}{2}-\sqrt{2}$ \\ 
5 & $b\bar{a}a$ & $\underline{\bar{a}}\ \bar{a}\bar{b}a\ 
   \underline{a}b\bar{a}$ & $x=-1+\frac{3\sqrt{2}}{4}$ \\
6 & $b\bar{a}\bar{a}$ & $\underline{\bar{a}}\ \bar{a}\bar{b}a\ 
   \underline{\bar{a}}\bar{b}a$ & $z=\frac{3}{2}-\sqrt{2}$ \\
\hline
\end{tabular}
\bigskip
\caption{Possible triples of consecutive (super-)tiles in the 
  TSM tiling, along with their inflations and 
  relative frequencies. For each triple given, there is also a 
  bar-swapped version, which contributes the same to the correlation.
  As given, the relative frequencies add up to $\frac{1}{2}$.
\label{tab:ttrip}}}
\end{table}

Let us now discuss how the different triples contribute. The first two
triples together contribute
$c(ab\bar{a},b\bar{a}a)+c(ab\bar{a},\bar{b}aa)$.  We see that the
contributions of the first two tiles of each triple cancel. What
remains is the anti-match of the third tiles.  Asymptotically, we get
an anti-match on a fraction $\lambda^{-1}$ of the total patch length,
which has to be weighted with the relative frequency of the
triples. The third and the fifth triple each contribute an anti-match
on the whole length of the patch, whereas the sixth triple contributes
a match. Somewhat more complicated is the contribution
$c(ab\bar{a},b\bar{a}a)$ of the fourth triple.  We get an anti-match
due to the third tile on a fraction $\lambda^{-1}$ of the patch, plus
$c(ab,b\bar{a})$. Inflating the latter once, we obtain
\[
  c(ab,b\bar{a}) \, = \,
   c(ab\bar{a}\bar{a},\bar{a}\bar{a}\bar{b}a) \, = \,
  \frac{(1+\lambda)\, c(b\bar{a},\bar{a} b)-2\lambda}{3\lambda+1}
  \, = \, \lambda^{-1}\bigl(c(b\bar{a},\bar{a}b)-\sqrt{2}\,\bigr),
\]
and, in a similar way,
\begin{align*}
   c(     b \bar{a}, \bar{a}b) &= 
       \lambda^{-1}\bigl(c(\bar{a}\bar{b}, 
       \bar{b}a)  + \sqrt{2}\,\bigr), \\
   c(\bar{a}\bar{b}, \bar{b}a) &= 
       \lambda^{-1}\bigl(c(\bar{b} a , a b)
      - \sqrt{2}\,\bigr), \\
   c(\bar{b} a , a b) &= 
   \lambda^{-1}\bigl(c( a b , b\bar{a}) + \sqrt{2}\,\bigr).
\end{align*}
These equations can be solved for $c(ab,b\bar{a})$, which gives
$c(ab,b\bar{a})=1-\sqrt{2}$. Putting everything together,
we get
\begin{equation}
    \lim_{n\to\infty}\nu^{}_\mathsf{tsm}(z_n) \, = \,
   \frac{\lambda\left(-2x\lambda^{-1}-y-z
   \frac{\lambda+(1+\lambda)(1-\sqrt{2})}
    {2\lambda+1}\right)+z-x}{\lambda(2x+y+z)+x+y}
    \, = \, 1-\sqrt{2} \ts ,
\end{equation}
where the relative frequencies $x$, $y$ and $z$ have been taken from 
Table~\ref{tab:ttrip}.
\end{proof}

\begin{theorem}
  The twisted silver mean dynamical system has two spectral
  components, both of which are singular. The dynamical spectrum from
  the even sector under the bar swap is pure point, whereas that
  from the odd sector is purely singular continuous.
\end{theorem}

\begin{proof}
  By Lemma~\ref{lem:specpure}, the dynamical spectrum in the even
  sector is pure point, whereas in the odd sector it is continuous and
  of pure spectral type. By Lemma~\ref{nu-non-decay}, the correlation
  measure $\sum_{z\in\varLambda-\varLambda} \nu^{}_\mathsf{tsm}(z)
  \delta_z$ does not decay to zero towards infinity, so that, by the
  Riemann--Lebesgue lemma and the fact that $\varLambda - \varLambda$
  is uniformly discrete for a Pisot inflation, its Fourier transform
  must be singular. This implies that there must be a singular
  continuous component in the diffraction spectrum, and hence in the
  dynamical spectrum of the odd sector. As that spectrum is of pure
  type, it must be purely singular continuous.
\end{proof}

\begin{remark}
  Starting from a tiling with pure point spectrum, we constructed
  another tiling with a mixed spectrum by splitting each tile into a
  barred and an unbarred variant, and modified the inflation such that
  a bar swap symmetry results. Tilings with mixed spectrum can also be
  obtained if each tile of a pure point tiling is split into $k$
  copies, and the inflation is modified such that a permutation
  symmetry acting on these $k$ copies results.  However, a spectral
  purity result for the continuous spectrum sector, as in
  Lemma~\ref{lem:specpure}, can only be expected if $k=2$.
\end{remark}

\begin{remark}
  The set $\varLambda_+$ of left endpoints of all un-barred tiles of a
  twisted silver mean tiling is a Meyer set (it is a relatively dense
  subset of a model set), which is linearly repetitive (it is a
  component of a primitive inflation tiling), and which has mixed
  spectrum of singular type. This shows that there are highly ordered
  Meyer sets with mixed spectrum.  In line with a result of Sinai,
  they have entropy $0$, in contrast to Meyer sets that can be thought
  of as model sets with an (a posteriori) thinning disorder of
  positive entropy. Such cases would have a mixed spectrum with a
  non-trivial pure point part and some absolutely continuous
  component. One example is provided by sets of the form $2 \ZZ \cup
  \vL$ where $\vL$ is a Bernoulli subset of $2 \ZZ + 1$ obtained by
  coin flipping, compare \cite[Sec.~11.2]{TAO}, and similar
  constructions on the basis of arbitrary model sets. Another example
  is provided by the random noble means inflation rules discussed in
  \cite{GL,Moll}.
\end{remark}

\section*{Acknowledgements}

We are grateful to Nicolae Strungaru and Mariusz Lemanzcyk for
discussions. This work was supported by the German Research Foundation
(DFG), within the CRC 701. We thank an anonymous reviewer for some
careful comments that helped to improve the presentation.

\end{document}